\documentclass[12pt]{amsart}
\usepackage{amsmath,amsfonts,amssymb,amsthm,amstext,pgf,graphicx,hyperref,verbatim,lmodern,ytableau,textcomp,color,young,youngtab,tikz,tikz-cd,bbold,cancel,enumitem}
\usepackage[mathscr]{euscript}
\definecolor{asparagus}{rgb}{0.0, 0.5, 0.0}
\makeatletter

\makeatletter
\newcommand\xlabel[2][]{\phantomsection\def\@currentlabelname{#1}\label{#2}}
\makeatother
\theoremstyle{plain}
\newtheorem{thm}{Theorem}[section]
\newtheorem{lem}[thm]{Lemma}
\newtheorem{prop}[thm]{Proposition}

\newtheorem{opqn}{Open Question}

\theoremstyle{definition}
\newtheorem{defn}{Definition}[section]

\newcommand{\group}{\mathfrak{G}}

\newcommand{\ynirr}{\mathcal{Y}_n(\widehat{G})}
\newcommand{\irrG}{\widehat{G}}

\newcommand{\B}{\mathcal{B}}
\DeclareMathOperator{\rate}{\bf{r}}
\DeclareMathOperator{\D}{Diag}
\DeclareMathOperator{\diag}{diag}
\DeclareMathOperator{\Sp}{Spec}
\DeclareMathOperator{\s}{sym}
\DeclareMathOperator{\transp}{\mathscr{T}}
\DeclareMathOperator{\T}{\bf{T}}
\DeclareMathOperator{\NYGD}{\mathscr{N}}
\DeclareMathOperator{\MYGD}{\mathscr{M}}
\DeclareMathOperator{\G}{\mathscr{G}}
\DeclareMathOperator{\V}{\mathscr{V}}
\DeclareMathOperator{\W}{\mathscr{W}}
\DeclareMathOperator{\Gnid}{\mathbf{1}}
\DeclareMathOperator{\Gid}{\mathbf{e}}

\DeclareMathOperator{\1}{\mathbf{id}}
\DeclareMathOperator{\Supp}{Supp}

\DeclareMathOperator{\trivial}{tr}
\DeclareMathOperator{\Tr}{trace}

\DeclareMathOperator{\dimension}{dim}
\DeclareMathOperator{\y}{\mathcal{Y}}
\DeclareMathOperator{\yn}{\mathcal{Y}_n}
\DeclareMathOperator{\yx}{\mathcal{Y}(X)}
\DeclareMathOperator{\ynx}{\mathcal{Y}_n(X)}

\DeclareMathOperator{\tab}{tab}

\title{}
\begin{document}
\title[Aldous-type spectral gap results for $G\wr S_n$]{Aldous-type spectral gap results for the complete monomial group}
\author{Subhajit Ghosh}
\address[Subhajit Ghosh]{Department of Mathematics, Bar-Ilan University, Ramat-Gan 5290002}
\email{gsubhajit@alum.iisc.ac.in}
\makeatletter
\@namedef{subjclassname@2020}{%
	\textup{2020} Mathematics Subject Classification}
\makeatother
\begin{abstract}
Let us consider the continuous-time random walk on $G\wr S_n$, the complete monomial group of degree $n$ over a finite group $G$, as follows: An element in $G\wr S_n$ can be multiplied (left or right) by an element of the form
\begin{itemize}
	\item $(u,v)_G:=(\Gid,\dots,\Gid;(u,v))$ with rate $x_{u,v}(\geq 0)$, or
	\item $(g)^{(w)}:=(\dots,\Gid,\hspace*{-0.65cm}\underset{\substack{\uparrow\\w\text{th position}}}{g}\hspace*{-0.65cm},\Gid,\dots;\1)$ with rate $y_w\alpha_g\; (y_w> 0,\;\alpha_g=\alpha_{g^{-1}}\geq 0)$,
\end{itemize} 
such that $\{(u,v)_G,\;(g)^{(w)}:x_{u,v}>0,\;y_w\alpha_g>0,\;1\leq u<v\leq n,\;g\in G,\;1\leq w\leq n\}$ generates $G\wr S_n$. We also consider the continuous-time random walk on $G\times\{1,\dots,n\}$ generated by one natural action of the elements $(u,v)_G,1\leq u<v\leq n$ and $(g)^{(w)},\;g\in G,1\leq w\leq n$ on $G\times\{1,\dots,n\}$ with the aforementioned rates. We show that the spectral gaps of the two random walks are the same. This is an analogue of the Aldous' spectral gap conjecture for the complete monomial group of degree $n$ over a finite group $G$.
\end{abstract}
\maketitle
\section{Introduction}\label{intro}
Let $S_n$ denote the symmetric group of permutations on $n$ letters $[n]:=\{1,\dots,n\}$. Also, let $\G=(V,E)$ be any graph with vertex set $V=[n]$. The (weighted) interchange process on $\G$ is defined as follows: Suppose the edges of graph $\G$ are equipped with Poisson clocks (alarm clocks that ring at time distributed as the exponential distribution); all clocks are independent. Let the edge $\{u,v\}\in E$ be equipped with the Poisson clock with rate $x_{u,v}$. Now place $n$ distinct marbles at the vertices of the graph, and whenever the clock at $\{u,v\}$ rings exchange the marbles at positions $u$ and $v$. Therefore, each marble does continuous-time random walk on the graph, but the walk for any two different marbles are independent. The process can be viewed as a continuous-time random walk on the symmetric group $S_n$, as the positions of the marbles at a given time are a permutation of their original positions. The Aldous' spectral gap conjecture says that \emph{the spectral gap for the interchange process on a finite connected graph is the same as the spectral gap of the continuous-time random walk on that graph}. Here, the connected graph is necessary for the irreducibility of the corresponding Markov chain. Recall that the smallest eigenvalue of the continuous-time Markov chain is zero. Moreover, the multiplicity of the smallest eigenvalue is one if and only if the Markov chain is irreducible. The \emph{spectral gap} of an irreducible continuous-time Markov chain is defined as the smallest positive eigenvalue of the chain. 

The Aldous' spectral gap conjecture was stated around 1992 as an open question in Aldous' homepage \cite{Ald._conj}. The original conjecture was for the finite unweighted (i.e., the rates of the Poisson clocks are the same) connected graphs. The conjecture had received a lot of attention since it was stated, and various special cases were established until it was proved to the complete generality (for finite weighted connected graphs) by Caputo et al. in 2010 (see \cite{CLR} and the references therein). Afterwards, several extensions have been established for the continuous-time random walk on $S_n$ with different generating sets and also on some other groups (for instance, see \cite{Cesi,PP}). 

In this article, we generalise the Aldous' spectral gap conjecture to the complete monomial group $G\wr S_n$ of degree $n$ over a finite group $G$. Our setup deals with oriented marbles. The orientations of each marble are indexed by the elements of the group $G$. A marble with orientation $x\;(\in G)$ is said to \emph{update its orientation using $g\;(\in G)$} if it changes its orientation from $x$ to $g\cdot x$. We now define the oriented interchange process as follows:
\begin{defn}[{Oriented interchange process}]\label{def:IP_Ori-Par}
	Let $\G=(V,E)$ be any graph with vertex set $V=[n]$. The edges and the vertices are equipped with Poisson clocks; all clocks are independent. Let the edge $\{u,v\}\in E$ be equipped with the Poisson clock with rate $x_{u,v}\;(\geq 0)$\footnote{We set $x_{u,v}=0$ when the edge $\{u,v\}$ is absent.}, and the vertex $w\in V$ be equipped with the Poisson clock with rate $y_w\;(>0)$. Now place $n$ distinct oriented marbles at the vertices of the graph. Whenever the clock at edge $\{u,v\}$ rings, exchange the marbles at positions $u$ and $v$; and whenever the clock at vertex $w$ rings, the marble updates its orientation using $g\;(\in G)$ with rate $\alpha_g\geq0$. Here, we assume that the subset $\{g\in G:\alpha_g>0\}$ generates the group $G$. 
\end{defn}
Each oriented marble in Definition \ref{def:IP_Ori-Par} performs a continuous-time random walk on $\G$. However, the walks for any two different oriented marbles are independent. The oriented interchange process can be viewed as a continuous-time random walk on the complete monomial group $G\wr S_n$, as the configuration (positions of the marbles and their orientations) at a given time can be determined by elements of $G\wr S_n$. The continuous-time random walk of one oriented marble can be explained as the \emph{continuous-time lamplighter random walk on $\G$ with $G$-valued lamps}. Let us now recall the continuous-time lamplighter random walk on $\G$ with $G$-valued lamps. Consider the graph $\G$ of Definition \ref{def:IP_Ori-Par} with the independent Poisson clocks equipped at the edges and the vertices. Suppose each vertex has a lamp with configurations indexed by the elements of $G$, abbreviated as $G$-valued lamp. A lamplighter at vertex $u$ moves to a neighbouring vertex $v$ when the clock at $\{u,v\}\in E$ rings and changes the lamp configuration at vertex $u$ when the clock at $u\in V$ rings.  If the configuration of a lamp is $x\in G$, then it changes to $g\cdot x\in G$ with rate $\alpha_g(\geq0)$; we assume that the set $\{g\in G:\alpha_g>0\}$ generates $G$. The oriented interchange process can also be explained in the lamplighter group setting with $n$ lamplighters at the vertices of $\G$. Let $u,v\in V$, the lamplighters at $u$ and $v$ exchange their positions whenever the clock at $\{u,v\}\in E$ rings; whenever the clock at vertex $w\in V$ rings, the lamplighter at $w\in V$ change the lamp configuration. If the configuration of a lamp is $x\in G$, then it changes to $g\cdot x\in G$ with rate $\alpha_g(\geq0)$; we assume that the set $\{g\in G:\alpha_g>0\}$ generates $G$. Therefore, each lamplighter performs continuous-time lamplighter random walk on $\G$ with $G$-valued lamps, and the walk of any two lamplighters are independent. The process can be viewed as a continuous-time random walk on $G\wr S_n$. We note that in the case of the process with oriented marbles, the generators of the random walk on $G\wr S_n$ act by multiplication on the left, whereas for the case with lamplighters, the multiplication is on the right. Here, we adopt the oriented marble that uses the \emph{left regular representation} (will be defined in Section \ref{sec:Preliminaries}); the lamplighter case can be dealt with by considering the \emph{right regular representation} (will be defined in Section \ref{sec:Preliminaries}) -- the analysis is the same. Our main theorem says: \emph{Let $\G$ be a connected graph with vertex set $[n]$. If $\alpha_g=\alpha_{g^{-1}}$, then the spectral gap of the oriented interchange process on $\G$ is the same as that of the process with one oriented marble}. Here, we have considered a connected graph to ensure the irreducibility of the associated (continuous-time) random walk. The formal statement is given below.
\begin{thm}\label{thm:main_thm}
	Let $G$ be a finite group, and $\G=(V,E)$ be any connected graph with vertex set $V=[n]$. Consider the oriented interchange process on $\G$ from Definition \ref{def:IP_Ori-Par}. If $\alpha_g=\alpha_{g^{-1}}$ for all $g\in G$, then the spectral gap of the oriented interchange process on $\G$ is the same as the spectral gap of the continuous-time lamplighter random walk on $\G$ with $G$-valued lamps.
\end{thm}
The proof will be given in Section \ref{sec:main_proof}. In upcoming work, we plan to explore the direction of ``long cycle formation in the oriented interchange process" to extend Kozma and Alons's result \cite{AK} to $G\wr S_n$ -- which necessarily requires Theorem \ref{thm:main_thm}. The main question of this article appeared in a personal communication with Gady Kozma.

Let us now discuss Aldous' spectral gap conjecture in the group algebra setting. The \emph{group algebra} $\mathbb{C}[\group]$ on a finite group $\group$ can be defined as the complex vector space consisting of all formal linear combinations of the elements of $\group$ with coefficients from $\mathbb{C}$. We note that $\group$ is a natural basis of $\mathbb{C}[\group]$. Given an action of $\group$ on a set $\mathscr{X}$, there is a canonical extension to the group algebra $\mathbb{C}[\group]$, which linearly extends the group action to $\mathbb{C}[\group]$. If $\mathbb{C}[\mathscr{X}]$ denotes the complex vector space of all formal linear combinations of the elements of $\mathscr{X}$, then the action of any group algebra element on $\mathscr{X}$ can be thought of as an operator on $\mathbb{C}[\mathscr{X}]$. The Aldous' spectral gap conjecture in the group algebra setting can be stated as follows: \emph{Let $(u,v)$ denote the transposition in $S_n$ interchanging $u$ and $v\;(1\leq u<v\leq n)$, and $\1$ denote the identity element of $S_n$. Consider the group algebra element 
\[\mathcal{S}=\displaystyle\sum_{1\leq u<v\leq n}x_{u,v}(u,v)\in\mathbb{C}[S_n],\;\;x_{u,v}\geq 0,\] 
such that $\Supp(\mathcal{S}):=\{(u,v):x_{u,v}>0,1\leq u<v\leq n\}$ generates $S_n$. Then, the smallest positive eigenvalue of the (left) multiplication action of 
\[\Delta(\mathcal{S})=\sum_{1\leq u<v\leq n}x_{u,v}\left(\1-(u,v)\right)\in\mathbb{C}[S_n]\]
on $S_n$ is the same as the smallest positive eigenvalue of the natural action of $\Delta(\mathcal{S})$ on the set $[n]$}. 

In recent work, Cesi \cite{Cesi} has proved an analogue of the Aldous' spectral gap conjecture for the \emph{hyperoctahedral group}, denoted $B_n$, defined as the group of bijections $\pi_B$ on $[\pm n]:=\{-n,\dots,-1,1,\dots,n\}$ satisfying $\pi_B(-i)=-\pi_B(i)$ for all $i\in[\pm n]$. Here, the group operation is the composition of bijective mappings. Any elements $\pi_B$ of $B_n$ is uniquely determined by its images on the set $[n]$ as $\pi_B(-i)=-\pi_B(i)$ for all $i\in[\pm n]$. Therefore, we may write $\pi_B$ in the window notation as $\pi_B=[\pi_B(1),\dots,\pi_B(n)]$. The hyperoctahedral group is isomorphic to $S_2\wr S_n$; an isomorphism can be defined on the generators as follows:
\begin{align*}
	(-j,j)_B:=[1,\dots,j-1,&\;\underset{\uparrow}{-j},j+1\dots,n] \longleftrightarrow (\1,\dots,\1,\;\underset{\uparrow}{(1,2)},\1,\dots,\1;\1),\text{ and }\\[-1ex]
	&j\text{th position.}\hspace*{3.5cm}j\text{th position.}\\&\\
	(i,j)_B:=[1,\dots,i-1,&\;\underset{\uparrow}{j},i+1\dots,j-1,\;\underset{\uparrow}{i},j+1\dots,n] \longleftrightarrow (\1,\dots,\1;(i,j)).\\[-1ex]
	&i\text{th position.}\hspace*{1cm}j\text{th position.}
\end{align*}
In this paper, we generalise Cesi's result \cite[Theorem 1.2]{Cesi} from $S_2\wr S_n$ to $G\wr S_n$; this is immediate from the statement of Theorem \ref{thm:main_thm} in the group algebra setting.

We now recall the definition of the wreath product $G\wr S_n$ of the finite group $G$ with $S_n$. The elements of $G\wr S_n$ are of the form $\left(g_1,\dots,g_n;\pi\right)\in G^n\times S_n$, and the group operation is given by
\[\left(g_1,\dots,g_n;\pi\right)\cdot\left(g_1',\dots,g_n';\sigma\right)=\left(g_1 g_{\pi^{-1}(1)}',\dots,g_n g_{\pi^{-1}(n)}';\pi\sigma\right).\]
Let us now set some notations that we use throughout the article. We denote the identity element of $G$ by $\Gid$, and the identity permutation in $S_n$ by $\1 $. Thus, $\Gnid=(\Gid,\dots,\Gid;\1)$ is the identity element of $G\wr S_n$. We denote $(\Gid,\dots,\Gid;\pi)$ simply by $\pi_G$. Also, for $1\leq j\leq n$, we use the notation $\left(g\right)^{(j)}$, to denote 
\begin{align*}
	(\Gid,\dots,\Gid,\;&\underset{\uparrow}{g},\Gid,\dots,\Gid;\1).\\[-1ex]
	&j\text{th position.}
\end{align*}
Let us also recall a natural action of $G\wr S_n$ on the set $G\times [n]$, defined by
\begin{equation}\label{eq:grp_action}
	\left(g_1,\dots,g_n;\pi\right)\cdot (h,i)=(g_{\pi(i)}\cdot h,\pi(i))\text{ for }	\left(g_1,\dots,g_n;\pi\right)\in G\wr S_n,\;(h,i)\in G\times [n].
\end{equation}
We are now in the position to state Theorem \ref{thm:main_thm} in the group algebra setting as follows: \emph{Let
	\begin{equation}\label{eq:group_alg_sett}
		\mathcal{G}_n:=\sum_{1\leq u<v\leq n}x_{u,v}(u,v)_G+\sum_{w=1}^{n}y_w\sum_{g\in G}\alpha_g\left(g\right)^{(w)}\in\mathbb{C}[G\wr S_n],
	\end{equation}
	$x_{u,v}\geq 0, y_w> 0$, and $\alpha_g=\alpha_{g^{-1}}>0$ for all $g\in G$, be such that 
	\begin{equation}\label{eq:generating-cond}
		\Supp(\mathcal{G}_n)=\big\{(u,v)_G,\;\left(g\right)^{(w)}:x_{u,v},\;\alpha_g>0,\;1\leq u<v\leq n,\;1\leq w\leq n,\;g\in G\big\}
	\end{equation} 
generates the group $G\wr S_n$. Then the smallest positive eigenvalue of the (left) multiplication action of 
\[\Delta(\mathcal{G}_n):=\sum_{1\leq u<v\leq n}x_{u,v}\left(\Gnid-(u,v)_G\right)+\sum_{w=1}^{n}y_w\sum_{g\in G}\alpha_g\left(\Gnid-\left(g\right)^{(w)}\right)\in\mathbb{C}[G\wr S_n]\]
on $G\wr S_n$ is the same as the smallest positive eigenvalues of the action \eqref{eq:grp_action} of $\Delta(\mathcal{G}_n)$ on the sets $G\times [n]$}. 

On a purely algebraic note, it is common to ask the question whether Theorem \ref{thm:main_thm} is true if we consider the reflections of type
\begin{align}\label{eq:gen_refl}
	(\Gid,\dots,\Gid,&\;\underset{\uparrow}{g},\Gid,\dots,\Gid,\;\underset{\uparrow}{g^{-1}},\Gid,\dots,\Gid;(i,j)),\;1\leq i<j\leq n,g\in G\\[-1ex]
	&\hspace*{-1cm}i\text{th position.}\hspace*{0.5cm}j\text{th position.}\nonumber
\end{align}
in the group algebra element \eqref{eq:group_alg_sett}. We explore this direction in Section \ref{sec:examples} with an example.

\emph{The organisation of this paper is as follows: We will devote Section \ref{sec:Preliminaries} to the necessary background on the representation theory of the finite group and set some notations. In Section \ref{sec:GwrSn}, we discuss the representation theory of $G\wr S_n$ and prepare the platform for the proof of Theorem \ref{thm:main_thm}. We will prove Theorem \ref{thm:main_thm} in Section \ref{sec:main_proof}. Finally, in Section \ref{sec:examples}, we will provide an example showing that Theorem \ref{thm:main_thm} do not hold if we consider reflections of the form \eqref{eq:gen_refl} in the group algebra element \eqref{eq:group_alg_sett} and state an open question for further study}.
\section{Notation and preliminaries}\label{sec:Preliminaries}
We divide this section into two subsections. In Subsection \ref{subsec:repn._theory_bkground}, we briefly recall some standard definitions and facts from the representation theory of finite groups to make the article self-contained; for more details, see \cite{Sagan,Serre}. In Subsection \ref{subsec:notations}, we set some notations for the rest of the paper.
\subsection{Representation theory background}\label{subsec:repn._theory_bkground}
Let $\V$ be a finite-dimensional complex vector space and GL$(\V)$ be the group of all invertible linear operators from $\V$ to itself under the composition of linear mappings.  Unless otherwise stated, all the vector spaces considered in this paper are finite-dimensional. Elements of GL$(\V)$ can be thought of as invertible matrices over $\mathbb{C}$. Let $\group$ be a finite group. Let $I_{\V}$ denote the identity element of GL$(\V)$ (i.e. the identity operator on $\V$), and $\mathfrak{e}$ denote the identity element of $\group$. A (complex) \emph{linear representation} $(\rho,\V)$ of $\group$ is a homomorphism $\rho:\group\rightarrow \text{GL}(\V)$, i.e., \[\rho(\mathfrak{g}_1\mathfrak{g}_2)=\rho(\mathfrak{g}_1)\rho(\mathfrak{g}_2)\text{ for all }\mathfrak{g}_1,\;\mathfrak{g}_2\in \group.\]
In particular, $\rho(\mathfrak{e})=I_{\V}$ and $\rho(\mathfrak{g}^{-1})=\rho(\mathfrak{g})^{-1},\;\mathfrak{g}\in \group$. The representation space $\V$ is called the \emph{$\group$-module} corresponding to the representation $\rho$. Given $\rho$, we simply say $\V$ is a representation of $\group$. Two useful examples are listed below:
\begin{itemize}
	\item Let $\V$ be one-dimensional. Then the representation $\trivial:\group\rightarrow \text{GL}(\V)$ defined by $\trivial(\mathfrak{g})\mapsto(v\mapsto v)$ for all $v\in \V$ and $g\in \group$ is known as the \emph{trivial representation} of $\group$. An alternative definition says that $\trivial(\mathfrak{g})=1$ for all $\mathfrak{g}\in\group$, as GL$(\V)$ is isomorphic to the multiplicative group of non-zero complex numbers when $\V$ is one-dimensional.
	\item Recall that the group algebra $\mathbb{C}[\group]=\{\sum_{\mathfrak{g}}c_\mathfrak{g}\mathfrak{g}\mid c_\mathfrak{g}\in\mathbb{C},\; \mathfrak{g}\in \group\}$ is the complex vector space of all formal linear combinations of the elements of $\group$ with complex coefficients. Then $\mathfrak{L}:\group\longrightarrow \text{GL}(\mathbb{C}[\group])$, the \emph{left regular representation} of $\group$ is defined by 
	\[\mathfrak{L}(\mathfrak{g})\left(\sum_{\mathfrak{h}\in\group}C_\mathfrak{h}\mathfrak{h}\right)=\sum_{\mathfrak{h}\in\group}C_\mathfrak{h}\mathfrak{g}\mathfrak{h},\; C_\mathfrak{h}\in\mathbb{C}.\]
	From now on, we denote the left regular representation of $G$ by $L$ and that of $G\wr S_n$ by $\mathscr{L}$. Similarly, the \emph{right regular representation} of $\group$, denoted $\mathfrak{R}:\group\longrightarrow \text{GL}(\mathbb{C}[\group])$, defined by 
	$\mathfrak{R}(\mathfrak{g})\left(\sum_{\mathfrak{h}\in\group}C_\mathfrak{h}\mathfrak{h}\right)=\sum_{\mathfrak{h}\in\group}C_\mathfrak{h}\mathfrak{h}\mathfrak{g}^{-1},\; C_\mathfrak{h}\in\mathbb{C}.$
\end{itemize}
The dimension of the vector space $\V$ is said to be the \emph{dimension} of the representation $\rho$ and is denoted by $d_{\rho}$. The trace of the matrix $\rho(\mathfrak{g})$ is said to be the \emph{character} value of $\rho$ at $\mathfrak{g}$ and is denoted by $\chi^{\rho}(\mathfrak{g})$. 
We have $\chi^{\rho}(\mathfrak{e})=d_{\rho}$, and $\chi^{\rho}(\mathfrak{g}^{-1})=\overline{\chi^{\rho}(\mathfrak{g})}$, the complex conjugate of $\chi^{\rho}(\mathfrak{g})$. A vector subspace $\W$ of $\V$ is said to be \emph{stable} (or \emph{invariant}) under $\rho$ if $\rho(\mathfrak{g})\left(\W\right)\subset \W$ for all $\mathfrak{g}$ in $\group$. The representation $\rho$ is \emph{irreducible} if $\V$ has no non-trivial proper stable subspace. For example, the trivial representation defined above is irreducible. Two representations $(\rho_1,\V_1)$ and $(\rho_2,\V_2)$ of $\group$ are said to be \emph{isomorphic} if there exists an invertible linear map $T:\V_1\rightarrow \V_2$ such that $T\circ\rho_1(\mathfrak{g})=\rho_2(\mathfrak{g})\circ T$ for all $\mathfrak{g}\in\group$; this is denoted by $\V_1\cong_{\group}\V_2$ or simply by $\V_1\cong \V_2$. 
\emph{Two representations of a finite group are isomorphic if and only if they have the same characters} \cite[Corollary 2]{Serre}. The \emph{direct sum of the representations} $(\rho_1,\V_1)$ and $(\rho_2,\V_2)$ is the representation $(\rho_1\oplus\rho_2):\group\rightarrow \text{GL}(\V_1\oplus \V_2)$ defined by, $(\rho_1\oplus\rho_2)(\mathfrak{g})(v_1\oplus v_2)=\rho_1(\mathfrak{g})(v_1)\oplus\rho_2(\mathfrak{g})(v_2)$ for $v_1\in \V_1,v_2\in \V_2$ and $\mathfrak{g}\in\group$.
The character of $\rho_1\oplus\rho_2$, denoted $\chi^{\rho_1\oplus\rho_2}$, given by $\chi^{\rho_1\oplus\rho_2}=\chi^{\rho_1}+\chi^{\rho_2}$. \emph{Every $($complex$)$ linear representation is a direct sum of irreducible representations} (\cite[Theorem 2 (Maschke's theorem)]{Serre}). Moreover, this decomposition is unique up to isomorphism of representation. Let $\group_1$ and $\group_2$ be two groups. Also, let $\V_1\otimes \V_2$ denote the tensor product of the (complex) vector spaces $\V_1$ and $\V_2$. Then the \emph{(external) tensor product of two representations} $\rho_1: \group_1\rightarrow \text{GL}(\V_1)$ and $\rho_2: \group_2\rightarrow \text{GL}(\V_2)$ is a representation of the direct product $\group_1\times\group_2$, denoted $(\rho_1\otimes\rho_2,\V_1\otimes \V_2)$, defined by,
\[(\rho_1\otimes\rho_2)(\mathfrak{g}_1,\mathfrak{g}_2)(v_1\otimes v_2)=\rho_1(\mathfrak{g}_1)(v_1)\otimes\rho_2(\mathfrak{g}_2)(v_2) \text{ for }v_i\in \V_i\text{ and }\mathfrak{g}_i\in\group_i,\;i=1,2.\]
The character $\chi^{\rho_1\otimes\rho_2}$ of $\rho_1\otimes\rho_2$ is given by $\chi^{\rho_1\otimes\rho_2}=\chi^{\rho_1}\chi^{\rho_2}$. \emph{The set of all irreducible $\group_1\times\group_2$-modules is given by $\{\V_1\otimes\V_2:\V_i\text{ is an irreducible }\group_i\text{-module},\; i=1,2\}$ \cite[Theorem 10]{Serre}}. Schur's lemma says that \emph{If a group algebra element $\mathfrak{g}\in\mathbb{C}[G]$ commutes with every element of the group $G$, then $\mathfrak{g}$ acts as a scalar on the irreducible $G$-modules} \cite[Proposition 5]{Serre}. We now define an inner product on the space of complex-valued functions defined on the finite group $\group$. Let $\phi,\psi:\group\rightarrow \mathbb{C}$ be two complex values functions defined on $\group$. Then the inner product $\langle\cdot,\cdot\rangle$ is defined by
\begin{equation}\label{eq:ChIP}
	\langle\phi,\psi\rangle:=\frac{1}{|\group|}\displaystyle\sum_{\mathfrak{g}\in \group}\phi(\mathfrak{g})\psi(\mathfrak{g}^{-1}).
\end{equation}
\begin{thm}[{\cite[Theorem 6]{Serre}}]\label{thm:ch2_irreducible_character_as_ONB}
	The characters corresponding to the non-isomorphic irreducible representations of $\group$ form an $\langle\cdot,\cdot\rangle$-orthonormal basis of the complex vector space of the class functions of $\group$.
\end{thm}
Let $\mathscr{H}$ be a subgroup of $\group$. The \emph{restriction} of the representation $\rho$ to $\mathscr{H}$ is denoted by $\rho\downarrow^\group_\mathscr{H}$ and is defined by $\rho\downarrow^\group_\mathscr{H}(\mathfrak{h}):=\rho(\mathfrak{h})$ for all $\mathfrak{h}\in \mathscr{H}$. If $\chi^{\rho}$ is the character of $\rho$, then the character of the restriction 
$\rho\downarrow^\group_\mathscr{H}$ is denoted by $\chi^\rho\downarrow^\group_\mathscr{H}$.
Let $\group$ be a finite group, and $\widehat{\group}$ be the set of equivalence classes (two representations are equivalent if they are isomorphic) of irreducible representations of $\group$. Then, from Theorem \ref{thm:ch2_irreducible_character_as_ONB}, the number of irreducible representations of a finite group is equal to the number of its conjugacy classes. If $\V\cong m_1\W_1\oplus\cdots\oplus m_{\ell}\W_{\ell}$ is the decomposition of the representation $(\rho,\V)$ into irreducible representations of $\rho$, then $\langle\chi^{\rho},\chi^{\rho}\rangle=m_1^2+\cdots+m_{\ell}^2$. Moreover, if $\chi^i$ denotes the irreducible character of $\W_i,\;1\leq i\leq\ell$, then $m_i=\langle\chi,\chi^i\rangle$ is called the multiplicity of $\W_i$ in the decomposition of $\V$. \emph{The regular (true for both left and right) representation of $\group$ decomposes into irreducible representations with multiplicity equal to their respective dimensions} \cite[p. 18, Corollary 1]{Serre}. Thus we have,
\begin{equation}\label{eq:Group_alg._decom.}
	\mathbb{C}[\group]\cong\underset{\rho\in\widehat{\group}}{\oplus}\;d_{\rho}\V^{\rho},
\end{equation}
where $\V^{\rho}$ is the irreducible $\group$-module corresponding to $\rho\in\widehat{\group}$ with dimension $d_{\rho}$. Also, from Theorem \ref{thm:ch2_irreducible_character_as_ONB} and \cite[Proposition 5]{Serre}, we have $|\group|=\displaystyle\sum_{\rho\in \widehat{\group}}d_{\rho}^2$.
\subsection{Continuous-time random walks on finite groups and the spectral gap}\label{subsec:notations}
A \emph{continuous-time random walk on a finite group $\group$ generated by a  rate function $\rate$} is the continuous-time Markov chain with state space $\group$ and the infinitesimal generator 
\[Q_{\group}(\rate):=\left(\rate(yx^{-1})-\delta_{x,y}\sum_{\mathfrak{g}\in \group}\rate(\mathfrak{g})\right)_{x,y\in\group},\]
where $\delta_{*,*}$ is the \emph{Kronecker delta}. Therefore, the transition from $x$ to $y$ happens with rate $\rate(yx^{-1})$, $x,y\in \group$. It can be easily seen that $-Q_{\group}(\rate)$ is the matrix of the action of
\[\displaystyle\sum_{\mathfrak{g}\in\group}\rate(\mathfrak{g})\left(\mathfrak{e}-\mathfrak{g}\right)\in\mathbb{C}[\group]\]
on $\mathbb{C}[\group]$ by multiplication on the left, with respect to the basis $\group$. Let us list down the eigenvalues of $-Q_{\group}(\rate)$, with possible repetitions according to their multiplicity, in weakly increasing order as follows:
\[0=\lambda_1(-Q_{\group}(\rate))\leq \lambda_2(-Q_{\group}(\rate))\leq\cdots\leq\lambda_{|\group|}(-Q_{\group}(\rate))\]
The random walk is irreducible \emph{if and only if the support of $\rate$, i.e., $\{\mathfrak{g}:\rate(\mathfrak{g})\neq 0\}$ generates the group $\group$} \cite[Proposition 2.3]{S}. In this case, the uniform distribution is the \emph{stationary} distribution of the random walk, and $\lambda_2(-Q_{\group}(\rate))>0$. The \emph{spectral gap} of the irreducible continuous-time random walk on $\group$ generated by $\rate$, denoted $\lambda_{\star}(-Q_{\group}(\rate))$, defined by $\lambda_{\star}(-Q_{\group}(\rate)):=\lambda_2(-Q_{\group}(\rate))$. We now set some useful notations for the rest of the paper.

Recall the group algebra $\mathbb{C}[\group]$, and set $\mathcal{W}=\sum_{\mathfrak{g}\in\group}\mathcal{W}_\mathfrak{g}\mathfrak{g}\in\mathbb{C}[\group]$. Then, the \emph{support} of $\mathcal{W}$, denoted $\Supp(\mathcal{W})$, defined by $\Supp(\mathcal{W}):=\{\mathfrak{g}\in\group:\mathcal{W}_\mathfrak{g}\neq 0\}$. Let us now define a canonical involution on $\mathbb{C}[\group]$ as follows:
\[\mathcal{W}=\sum_{\mathfrak{g}\in\group}\mathcal{W}_\mathfrak{g}\mathfrak{g}\mapsto\mathcal{W}^{\dagger}:=\sum_{\mathfrak{g}\in\group}\overline{\mathcal{W}}_\mathfrak{g}\mathfrak{g}^{-1},\text{ here }\overline{\mathcal{W}}_\mathfrak{g}\text{ is the complex conjugate of }\mathcal{W}_\mathfrak{g}.\]
The group algebra element $\mathcal{W}\in\mathbb{C}[\group]$ is called \emph{symmetric} if $\mathcal{W}=\mathcal{W}^{\dagger}$, and it is called \emph{non-negative} if $\mathcal{W}_{\mathfrak{g}}\geq 0$ for all $\mathfrak{g}\in\group$. Let us set
\begin{align}
	\mathbb{C}[\group]^{(\s)}&:=\{\mathcal{W}\in\mathbb{C}[\group]:\mathcal{W}\text{ is symmetric}\},\label{eq:Re-Sym_Gp-Alg0}\\
	\mathbb{R}_+[\group]^{(\s)}&:=\{\mathcal{W}\in\mathbb{C}[\group]:\mathcal{W}\text{ is symmetric and non-negative}\}.\label{eq:Re-Sym_Gp-Alg}
\end{align} 
Given a representation $(\rho,\V)$ of $\group$, we define the operator $\Delta_{\group}(\mathcal{W},\rho)$ on $\V$ as follows:
\begin{equation}\label{eq:delta-op}
	\Delta_{\group}(\mathcal{W},\rho):=\sum_{\mathfrak{g}\in\group}\mathcal{W}_\mathfrak{g}\left(I_{\V}-\rho(\mathfrak{g})\right).
\end{equation}
The eigenvalues of the operator $\Delta_{\group}(\mathcal{W},\rho)$ are real if $\mathcal{W}$ is symmetric, the eigenvalues are real non-negative when $\mathcal{W}$ is symmetric and non-negative \cite[Section 2]{Cesi-octopus}. For symmetric $\mathcal{W}\in\mathbb{C}[\group]$, we label the eigenvalues of $\Delta_{\group}(\mathcal{W},\rho)$ in weakly increasing order, with possible repetitions according to their multiplicity, as follows:
\[\lambda_1\left(\Delta_{\group}(\mathcal{W},\rho)\right)\leq \lambda_2\left(\Delta_{\group}(\mathcal{W},\rho)\right)\leq\cdots\leq\lambda_{d_{\rho}}\left(\Delta_{\group}(\mathcal{W},\rho)\right),\]
where $d_{\rho}$ is the dimension of the representation $\rho$. Also, for $\mathcal{W}\in\mathbb{R}_+[\group]^{(\s)}$, the \emph{spectral gap} of the pair $(\mathcal{W},\rho)$, denoted $\psi_{\group}(\mathcal{W},\rho)$, defined as 
\begin{equation}\label{eq:sp-gap-rep}
	\psi_{\group}(\mathcal{W},\rho):=\min\{\lambda>0:\lambda\text{ is an eigenvalue of }\Delta_{\group}(\mathcal{W},\rho)\},
\end{equation}
with the convention that $\min\emptyset=+\infty$. In particular, if the multiplicity of the trivial representation in $\rho$ is $m$, then $\psi_{\group}(\mathcal{W},\rho)\geq\lambda_{m+1}\left(\Delta_{\group}(\mathcal{W},\rho)\right)$ as $\Delta_{\group}(\mathcal{W},\rho)$ has exactly $m$ trivial eigenvalues.  The \emph{spectral gap} of $\mathcal{W}\in\mathbb{R}_+[\group]^{(\s)}$, denoted $\psi_{\group}(\mathcal{W})$, defined as follows:
\[\psi_{\group}(\mathcal{W}):=\inf\{\psi_{\group}(\mathcal{W},\rho):\rho\text{ is a representation of }\group\}.\]
Although the definition of $\psi_{\group}(\mathcal{W})$ uses all the representations of $\group$, it is sufficient to look only at the irreducible representations, thanks to Maschke's theorem. Therefore, we have
\begin{equation}\label{eq:sp-gap-rep1}
	\psi_{\group}(\mathcal{W}):=\min\{\psi_{\group}(\mathcal{W},\rho):\rho\in\widehat{\group}\}.
\end{equation}
Recall that $\mathfrak{L}$ be the left regular representation of $\group$. Therefore, the decomposition \eqref{eq:Group_alg._decom.} implies that $\psi_{\group}(\mathcal{W})=\psi_{\group}(\mathcal{W},\mathfrak{L})\geq\lambda_2\left(\Delta_{\group}(\mathcal{W},\mathfrak{L})\right)$, with equality if and only if $\Supp\left(\mathcal{W}\right)$ generates $\group$, for all $\mathcal{W}\in\mathcal{W}\in\mathbb{R}_+[\group]^{(\s)}$. Also, recall the continuous-time irreducible random walk on $\group$ generated by the rate function $\rate$. Then, we have, $\displaystyle\sum_{\mathfrak{g}\in\group}\rate(\mathfrak{g})\mathfrak{g}\in\mathbb{R}_+[\group]^{(\s)},\;-Q_{\group}(\rate)=\Delta_{\group}\left(\displaystyle\sum_{\mathfrak{g}\in\group}\rate(\mathfrak{g})\mathfrak{g},\mathfrak{L}\right)$, and hence 
	$\psi_{\group}\left(\displaystyle\sum_{\mathfrak{g}\in\group}\rate(\mathfrak{g})\mathfrak{g}\right)=\lambda_{\star}\left(-Q_{\group}(\rate)\right)$.
\section{The representation theory of $G\wr S_n$}\label{sec:GwrSn}
In this section, we prepare the platform for the proof of Theorem \ref{thm:main_thm}. We first revisit the representation theory of the complete monomial group $G\wr S_n$. We will define the necessary concepts and state the useful results; we refer to the exposition \cite{MS} for more details on the representation theory of $G\wr S_n$. We prove some results that will be used in Section \ref{sec:main_proof}. The main goal of this section is to obtain the decomposition of the representation defined 
by the action \eqref{eq:grp_action} into irreducible $G\wr S_n$-modules.

A \emph{partition} of a positive integer $n$, denoted $\eta\vdash n$, defined as a weakly decreasing (finite) sequence $(\eta_1,\cdots,\eta_r)$ of positive integers such that $\sum_{i=1}^{r}\eta_i=n$. The \emph{Young diagram} of the partition $\eta=(\eta_1,\cdots,\eta_r)$ is a left-justified arrangement of $r$ rows of boxes with $\eta_i$ boxes in the $i$th row, $1\leq i\leq r$. For example, there are five partitions of the positive integer $4$ viz. $(4)$, $(3,1)$, $(2,2)$, $(2,1,1)$ and $(1,1,1,1)$; Figure \ref{fig: Yng_diag_with_4_boxes} shows the corresponding Young diagrams.
\begin{figure}[h]
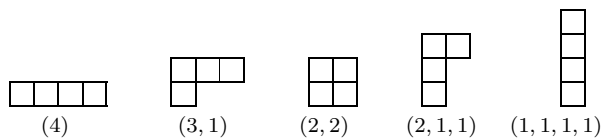

	\centering
	\tiny{
		$\begin{array}{cclll}
		\yng(4)&\hspace{0.5cm}\yng(3,1)& \hspace{0.5cm}\yng(2,2) & \hspace{0.5cm}\yng(2,1,1) & \hspace{0.75cm}\yng(1,1,1,1)\\
		(4)\;&\;\quad(3,1)&\quad\;(2,2)&\quad(2,1,1)&\;(1,1,1,1)
	\end{array}$}
	\caption{Young diagrams with $4$ boxes.}\label{fig: Yng_diag_with_4_boxes}
\end{figure}
 The set of all Young diagrams (there is a unique Young diagram with zero boxes) is denoted by $\y$, and the set of all Young diagrams with $n$ boxes is denoted by $\yn$. For example, elements of $\mathcal{Y}_4$ are shown in Figure \ref{fig: Yng_diag_with_4_boxes}. Given a finite set $X$, we define 
 \[\yx:=\{\mu:\mu\text{ is a map from }X\text{ to }\y\}.\]
 Let $\eta\in\y$, and $|\eta|$ denote the number of boxes in the Young diagram $\eta$. Also, let $n$ be a fixed positive integer. We define $\ynx:=\{\mu\in\yx:||\mu||=n\}$; here we have used $||\mu||:=\sum_{x\in X}|\mu(x)|$. Let $\widehat{G}$ denote the (finite) set of all non-isomorphic irreducible representations of $G$. Given $\sigma\in\widehat{G}$, we denote the corresponding irreducible $G$-module by $W^{\sigma}$. Elements of $\y(\widehat{G})$ are called \emph{Young $G$-diagrams}, and elements of $\yn(\widehat{G})$ are called \emph{Young $G$-diagrams with $n$ boxes}.
\begin{figure}[h]
	\tiny{
	\centering
	$\mu:=\left(\mu(\sigma_1),\mu(\sigma_2),\dots,\mu(\sigma_{10})\right)=\left(\begin{array}{c}\yng(3,2)\end{array}\hspace*{-0.75ex},\begin{array}{c}\yng(1,1)\end{array},\emptyset,\;\emptyset,\;\emptyset,\;\emptyset,\;\emptyset,\begin{array}{c}\yng(1,1)\end{array},\;\emptyset,\begin{array}{c}\yng(1)\end{array}\right)$}
	\caption{An element of $\mathcal{Y}_{10}(\widehat{\mathbb{Z}}_{10})$. Here $\widehat{\mathbb{Z}}_{10}:=\{\sigma_i:1\leq i\leq 10\}$}.\label{fig:Young_G_diagram}
\end{figure}
For example, if $n=10$ and $G=\mathbb{Z}_{10}$ (the additive group of integers modulo $10$), then Figure \ref{fig:Young_G_diagram} shows a Young $\mathbb{Z}_{10}$-diagram with $10$ boxes. Given $\eta\in\y$, a \emph{Young tableau} of shape $\eta$ is obtained by taking the Young diagram $\eta$ and filling its $|\eta|$ boxes (bijectively) with the numbers $1,2,\dots,|\eta|$. A Young tableau is said to be \emph{standard} if the numbers in the boxes strictly increase along each row and each column of the Young diagram of $\eta$. The set of all standard Young tableaux of shape $\eta$ is denoted by $\tab(\eta)$. Elements of $\tab((3,1))$ are listed in Figure \ref{fig: Stab_of_shape_(3,1)}.
\begin{figure}[h]
	\centering
	\tiny{
	$\begin{array}{ccc}
		\young({{\substack{1}}}{{\substack{2}}}{{\substack{3}}},{{\substack{4}}}) &\quad\young({{\substack{1}}}{{\substack{2}}}{{\substack{4}}},{{\substack{3}}}) & \quad \young({{\substack{1}}}{{\substack{3}}}{{\substack{4}}},{{\substack{2}}})
	\end{array}$}
	\caption{Standard Young tableaux of shape $(3,1)$.}\label{fig: Stab_of_shape_(3,1)}
\end{figure}
 Let $\mu\in\y(\widehat{G})$. Then, a \emph{Young $G$-tableau} of shape $\mu$ is obtained by taking the Young $G$-diagram $\mu$ and filling its $||\mu||$ boxes (bijectively) with the numbers $1,2,\dots,||\mu||.$ A Young $G$-tableau is said to be \emph{standard} if the numbers in the boxes strictly increase along each row and each column of all Young diagrams occurring in $\mu$. Given $\mu\in\yn(\widehat{G})$, the set of all standard Young $G$-tableaux of shape $\mu$ is denoted by $\tab_{G}(n,\mu)$. Let us set $\tab_G(n):=\underset{\mu\in\yn(\widehat{G})}{\cup}\tab_{G}(n,\mu)$. For example, an element of $\tab_{\mathbb{Z}_{10}}\left(10,\mu\right)$ is shown in Figure \ref{fig:Young_G_tableau}, the shape $\mu\;\left(\in\y_{10}(\widehat{\mathbb{Z}}_{10})\right)$ is given in Figure \ref{fig:Young_G_diagram}.
\begin{figure}[h]
	\tiny{
	\centering
	$\mu\rightsquigarrow\left(
	\begin{array}{c}\young({{\substack{4}}}{{\substack{6}}}{{\substack{9}}},{{\substack{7}}}{{\substack{10}}})\end{array}\hspace*{-0.75ex}, \begin{array}{c}\young({{\substack{1}}},{{\substack{2}}})\end{array}\hspace*{-0.75ex},\;\emptyset,\;\emptyset,\;\emptyset,\;\emptyset,\;\emptyset, \begin{array}{c}\young({{\substack{3}}},{{\substack{8}}})\end{array}\hspace*{-0.75ex},\;\emptyset,\begin{array}{c}\young({{\substack{5}}})\end{array}
	\right)$}
	\caption{A standard Young $\mathbb{Z}_{10}$-tableaux of shape $\mu$, defined in Figure \ref{fig:Young_G_diagram}.}\label{fig:Young_G_tableau}
\end{figure}
\begin{defn}\label{def:b_(i)+r_T(i)}
	Let $T\in\tab_G(n)$ and $i\in[n]$. \emph{If $i$ appears in the Young diagram $\mu(\sigma)$, where $\mu$ is the shape of $T$ and $\sigma\in\widehat{G}$, we write $r_{T}(i)=\sigma$}. For the example given in Figure \ref{fig:Young_G_tableau}, we have $r_{T}(1)=\sigma_2$, $r_{T}(2)=\sigma_2$, $r_{T}(3)=\sigma_8$, $r_{T}(4)=\sigma_1$, $r_{T}(5)=\sigma_{10}$, $r_{T}(6)=\sigma_{1}$, $r_{T}(7)=\sigma_{1}$, $r_{T}(8)=\sigma_8$, $ r_{T}(9)=\sigma_1$, $r_T(10)=\sigma_1$. The \emph{content} of a box in row $p$ and column $q$ of a Young diagram is the integer $q-p$. \emph{Let $b_T(i)$ be the box in $\mu(\sigma)$, with the number $i$ resides and $c(b_T(i))$ denote the content of the box $b_T(i)$}. For the example given in Figure \ref{fig:Young_G_tableau}, we also have $c(b_T(1))=0$, $c(b_T(2))=-1$, $c(b_T(3))=0$, $c(b_{T}(4))=0$, $c(b_{T}(5))=0$, $c(b_T(6))=1$, $c(b_T(7))=-1$, $c(b_{T}(8))=-1$, $c(b_{T}(9))=2$, $c(b_T(10))=0$.
\end{defn}
The irreducible representation of $G\wr S_n$ can be parametrised by elements of $\mathcal{Y}_n(\widehat{G})$ \cite[Lemma 6.2 and Theorem 6.4]{MS}. Given $\mu\in\y(\widehat{G})$ and $\sigma\in\widehat{G},\;\mu\downarrow_{\sigma}$ denotes the set of all Young $G$-diagrams obtained from $\mu$ by removing one of the inner corners (a box, whose removal leaves a valid Young diagram) in the Young diagram $\mu(\sigma)$; see Figure \ref{fig:branching-diagram} for an example.
\begin{figure}[h]
	\centering
	\tiny{
	$\begin{array}{c}
			\mu\downarrow_{\sigma_1}=\Bigg\{\left(\begin{array}{c}\yng(2,2)\end{array}\hspace*{-0.75ex},\begin{array}{c}\yng(1,1)\end{array},\emptyset,\;\emptyset,\;\emptyset,\;\emptyset,\;\emptyset,\begin{array}{c}\yng(1,1)\end{array},\;\emptyset,\begin{array}{c}\yng(1)\end{array}\right),\left(\begin{array}{c}\yng(3,1)\end{array}\hspace*{-0.75ex},\begin{array}{c}\yng(1,1)\end{array},\emptyset,\;\emptyset,\;\emptyset,\;\emptyset,\;\emptyset,\begin{array}{c}\yng(1,1)\end{array},\;\emptyset,\begin{array}{c}\yng(1)\end{array}\right)\Bigg\}\\
		\mu\downarrow_{\sigma_2}=\Bigg\{\left(\begin{array}{c}\yng(3,2)\end{array}\hspace*{-0.75ex},\begin{array}{c}\yng(1)\end{array},\emptyset,\;\emptyset,\;\emptyset,\;\emptyset,\;\emptyset,\begin{array}{c}\yng(1,1)\end{array},\;\emptyset,\begin{array}{c}\yng(1)\end{array}\right)\Bigg\}\\
			\mu\downarrow_{\sigma_8}=\Bigg\{\left(\begin{array}{c}\yng(3,2)\end{array}\hspace*{-0.75ex},\begin{array}{c}\yng(1,1)\end{array},\emptyset,\;\emptyset,\;\emptyset,\;\emptyset,\;\emptyset,\begin{array}{c}\yng(1)\end{array},\;\emptyset,\begin{array}{c}\yng(1)\end{array}\right)\Bigg\}\\
		\mu\downarrow_{\sigma_{10}}=\Bigg\{\left(\begin{array}{c}\yng(3,2)\end{array}\hspace*{-0.75ex},\begin{array}{c}\yng(1,1)\end{array},\emptyset,\;\emptyset,\;\emptyset,\;\emptyset,\;\emptyset,\begin{array}{c}\yng(1,1)\end{array},\;\emptyset,\;\emptyset\right)\Bigg\}\;\quad\;
	\end{array}$}
	\caption{$\mu\downarrow_{\sigma}$ for non empty $\mu({\sigma})$, where $\mu\left(\in\mathcal{Y}_{10}(\widehat{\mathbb{Z}}_{10})\right)$ is defined in Figure \ref{fig:Young_G_diagram}.}\label{fig:branching-diagram}\vspace*{-1ex}
\end{figure}
The branching rule \cite[Theorem 6.6]{MS} of the pair $G\wr S_{n-1}\subseteq G\wr S_n$ is stated below.
\begin{thm}[{\cite[Theorem 6.6 (Branching rule)]{MS}}]\label{thm:Branching-rule}
	Let $V^{\mu}$ (respectively $V^{\lambda}$) denote the irreducible $G\wr S_n$-module (respectively $G\wr S_{n-1}$-module) indexed by $\mu\in\yn(\widehat{G})$ (respectively $\lambda\in\mathcal{Y}_{n-1}(\widehat{G})$). Then,
	\begin{equation}\label{eq:GwrS-branching}
		V^{\mu}\big\downarrow_{G\wr S_{n-1}}^{G\wr S_{n}}\cong\underset{\sigma\in\widehat{G}:\;\mu(\sigma)\neq\emptyset}{\oplus}\dimension(W^{\sigma})\left(\underset{\lambda\in\mu\downarrow_{\sigma}}{\oplus}V^{\lambda}\right),
	\end{equation}
	where $W^{\sigma}$ is the irreducible $G$-module indexed by $\sigma$. 
\end{thm}
To illustrate \eqref{eq:GwrS-branching}, consider $\mu\in\y_{10}(\widehat{\mathbb{Z}}_{10})$ from Figure \ref{fig:Young_G_diagram}, and recall $\mu\downarrow_{\sigma_i}$ $(i=1,2,8,10)$ from Figure \ref{fig:branching-diagram}. Using the same notation for a singleton set and its element, we have
	\[V^{\mu}\big\downarrow^{\mathbb{Z}_{10}\wr S_{10}}_{\mathbb{Z}_{10}\wr S_9}\cong\dimension(W^{\sigma_1})\left(\underset{\lambda\in\mu\downarrow_{\sigma_1}}{\oplus}V^{\lambda}\right)\oplus\dimension(W^{\sigma_2})V^{\mu\downarrow_{\sigma_2}}\oplus\dimension(W^{\sigma_8})V^{\mu\downarrow_{\sigma_8}}\oplus\dimension(W^{\sigma_{10}})V^{\mu\downarrow_{\sigma_{10}}}.\]
Let $\mathcal{H}_{i,n}(G)$ be the subgroup $\{(g_1,\dots,g_n,\pi)\in G\wr S_n:\pi(j)=j\text{ for }i+1\leq j\leq n\}$ of $G\wr S_n$, $0\leq i\leq n$. In particular $\mathcal{H}_{0,n}(G)=\mathcal{H}_{1,n}(G)=G^n$ and $\mathcal{H}_{n,n}(G)=G\wr S_n$. The subgroup $\mathcal{H}_{i,n}(G)$ is isomorphic to $G\wr S_i\times G^{n-i}$ (direct product of $G\wr S_i$ and $G^{n-i}$) 
by the isomorphism
\[(g_1,\dots,g_i,g_{i+1},\dots,g_n;\pi)\in\mathcal{H}_{i,n}(G)\longleftrightarrow\left((g_1,\dots,g_i;\pi),(g_{i+1},\dots,g_n)\right)\in G\wr S_i\times G^{n-i}.\]
Here, we have used the same notation $\pi$ for permutations of $S_i$ and $S_n$, as $\pi\in S_i$ means $\pi(j)=j$ for $i+1\leq j\leq n$. The irreducible $G\wr S_i\times G^{n-i}$-modules are given by the tensor products of the irreducible $G\wr S_i$-modules and the irreducible $G^{n-i}$-modules \cite[Theorem 10]{Serre}. Therefore, we may parametrise the irreducible representations of $\mathcal{H}_{i,n}(G)$ by elements of $\y_i(\widehat{G})\times\widehat{G}^{n-i}$. The branching rule of the pair $\mathcal{H}_{i-1,n}(G)\subseteq \mathcal{H}_{i,n}(G)$ is given as follows: Let $V^{(\mu,\sigma_{i+1},\sigma_{i+2},\dots,\sigma_{n})}$ (respectively $V^{(\lambda,\sigma,\sigma_{i+1},\sigma_{i+2},\dots,\sigma_{n})}$) denote the irreducible $\mathcal{H}_{i,n}(G)$-module (respectively $\mathcal{H}_{i-1,n}(G)$-module) indexed by $(\mu,\sigma_{i+1},\sigma_{i+2},\dots,\sigma_{n})\in\y_i(\widehat{G})\times\widehat{G}^{n-i}$ (respectively $(\lambda,\sigma,\sigma_{i+1},\sigma_{i+2},\dots,\sigma_{n})\in\y_{i-1}(\widehat{G})\times\widehat{G}^{n-i+1}$). Then,
	\begin{equation}\label{eq:H_i,n-branching}
		V^{(\mu,\sigma_{i+1},\sigma_{i+2},\dots,\sigma_{n})}\big\downarrow_{\mathcal{H}_{i-1,n}(G)}^{\mathcal{H}_{i,n}(G)}\cong\underset{\sigma\in\widehat{G}:\;\mu(\sigma)\neq\emptyset}{\oplus}\left(\underset{\lambda\in\mu\downarrow_{\sigma}}{\oplus}V^{(\lambda,\sigma,\sigma_{i+1},\sigma_{i+2},\dots,\sigma_{n})}\right).
	\end{equation}
	In particular, $V^{(\mu)}=V^{\mu}$ (irreducible $G\wr S_n$-module), for $i=n$. To illustrate \eqref{eq:H_i,n-branching}, consider $\mu\in\y_{10}(\widehat{\mathbb{Z}}_{10})$ from Figure \ref{fig:Young_G_diagram}, and recall $\mu\downarrow_{\sigma_i}$ $(i=1,2,8,10)$ from Figure \ref{fig:branching-diagram}. Then
	\[V^{(\mu)}\big\downarrow_{\mathcal{H}_{9,10}(\mathbb{Z}_{10})}^{\mathcal{H}_{10,10}(\mathbb{Z}_{10})}\cong\left(\underset{\lambda\in\mu\downarrow_{\sigma_1}}{\oplus}V^{(\lambda,\sigma_1)}\right)\oplus V^{(\mu\downarrow_{\sigma_2},\sigma_2)}\oplus V^{(\mu\downarrow_{\sigma_8},\sigma_8)}\oplus V^{(\mu\downarrow_{\sigma_{10}},\sigma_{10})}.\]
	Here, we have used the same notation for a singleton set and its element.

Let $\mu\in\widehat{\mathcal{H}_{n,n}}(G)$ and consider the irreducible $\mathcal{H}_{n,n}(G)$-module $V^{\mu}$ (the space for the representation $\mu$). Since the branching is simple (recall \eqref{eq:H_i,n-branching}), the decomposition into irreducible $\mathcal{H}_{n-1,n}(G)$-modules is given by 
\[V^{\mu}=\underset{\lambda}{\oplus}V^{\lambda},\]
where the sum is over all $\lambda\in \widehat{\mathcal{H}_{n-1,n}}(G)$, with $\lambda\nearrow\mu$ (i.e. there is an edge from $\lambda$ to $\mu$ in the branching multi-graph), is canonical. Here we note that $\mu\in\yn(\widehat{G})$ and $\lambda\in\y_{n-1}(\widehat{G})\times\widehat{G}$. Iterating this decomposition of $V^{\mu}$ into irreducible $\mathcal{H}_{1,n}(G)$-submodules, i.e.,
\begin{equation}\label{eq:GZ-}
	V^{\mu}=\underset{T}{\oplus}V_{T},
\end{equation}
where the sum is over all possible chains $T=\mu_1\nearrow\mu_2\nearrow\dots\nearrow\mu_n$ with $\mu_i\in \widehat{\mathcal{H}_{i,n}}(G)$ and $\mu_n=\mu$. We call \eqref{eq:GZ-} the \emph{Gelfand-Tsetlin} decomposition of $V^{\mu}$ and each $V_T$ in \eqref{eq:GZ-} a \emph{Gelfand-Tsetlin} subspace of $V^{\mu}$. We note that if $0\neq v_T\in V_T,\text{ then }\mathbb{C}[\mathcal{H}_{i,n}(G)]v_T=V^{\mu_i}$ from the definition of $V_T$.
\begin{thm}[{\cite[Theorem 6.5]{MS}}]\label{thm:GT-decomposition}
	Let $\mu\in\yn(\widehat{G})$. Then, we may index the Gelfand-Tsetlin subspaces of $V^{\mu}$ by standard Young $G$-tableaux  of shape $\mu$ and write the Gelfand-Tsetlin decomposition as 
	\[V^{\mu}=\underset{T\in\tab_{G}(n,\mu)}{\oplus}V_T,\]
	where each $V_T$ is closed under the action of $G^n$ and as a $G^n$-module, is isomorphic to the irreducible $G^n$-module 
	\[W^{r_T(1)}\otimes W^{r_T(2)}\otimes\dots\otimes W^{r_T(n)}.\]
	Here $W^{r_T(i)}$ is the irreducible $G$-module indexed by $r_T(i)$; recall $r_T(i)$ from Definition \ref{def:b_(i)+r_T(i)}.
\end{thm}
\begin{thm}[{\cite[Theorem 6.7]{MS}}]\label{thm:dimen of irr G_n-modules}
	Let $\mu\in\yn(\widehat{G})$. Write the elements of $\widehat{G}$ as $\{\sigma_1,\dots,\sigma_t\}$ and set $\mu^{(i)}=\mu(\sigma_i),\;m_i=|\mu^{(i)}|,d_i=\dimension(W^{\sigma_i})$ for each $1\leq i\leq t$. Then 
	\[\dimension(V^{\mu})=\binom{n}{m_1,\dots,m_t}f^{\mu^{(1)}}\cdots f^{\mu^{(t)}}d_1^{m_1}\cdots d_t^{m_t}.\]
	Here $f^{\mu^{(i)}}$ denotes the number of standard Young tableau of shape $\mu^{(i)}$ for each $1\leq i\leq t$.
\end{thm}
Let $\sigma\in\irrG$. Fix a $G$-invariant inner product on $W^{\sigma}$, the irreducible $G$-module indexed by $\sigma$. Also, let $B^{\sigma}$ be any orthonormal basis of $W^{\sigma}$. For $\rho:=(\rho_1,\dots \rho_n)\in\left(\irrG\right)^n$, we use the notation $B^{\rho}$ to denote the basis $B^{\rho_1}\otimes\cdots\otimes B^{\rho_n}:=\{v_1\otimes\cdots\otimes v_n:v_i\in B^{\rho_i}, 1\leq i\leq n\}$ of the irreducible $G^n$-module $W^{\rho}:=W^{\rho_1}\otimes\cdots\otimes W^{\rho_n}$. The basis $B^{\rho}$ is orthonormal under the inner product obtained by multiplying the component inner products.  For $1\leq i< n-1$, we define the \emph{switch operator}
\[\tau_{i,\rho}:W^{\rho_1}\otimes\cdots\otimes W^{\rho_n}\longrightarrow W^{\rho_1}\otimes\cdots\otimes W^{\rho_{i-1}}\otimes W^{\rho_{i+1}}\otimes W^{\rho_i}\otimes W^{\rho_{i+2}}\otimes\cdots\otimes W^{\rho_n}\]
by switching the $i$th and $(i+1)$th factor, i.e., for $w_1\otimes\cdots\otimes w_n\in W^{\rho_1}\otimes\cdots\otimes W^{\rho_n}$,
\[\tau_{i,\rho}(w_1\otimes\cdots\otimes w_n)=w_1\otimes\cdots\otimes w_{i-1}\otimes w_{i+1}\otimes w_i\otimes w_{i+2}\otimes\cdots\otimes w_n.\]
Recall $r_T(i)$, where $T\in\tab_G(n,\mu)$ and $i\in[n]$, from Definition \ref{def:b_(i)+r_T(i)}. Therefore, $B^{r_T}$ is a basis of $W^{r_T}$, where $r_T:=(r_T(1),\dots,r_T(n))$. If $r_T(i)=r_T(i+1)$, then we let $M_{i,r_T}$ be the matrix of the switch operator $\tau_{i,r_T}$ on $W^{\rho}$ with respect to the basis $B^{r_T}$. For $T\in\tab_G(n,\mu)$, there is a $G^n$-linear isomorphism from $W^{r_T}$ to the Gelfand-Tsetlin subspace $V_T$ that takes the basis $B^{r_T}$ (of $W^{r_T}$) to the basis $\B_T$ (of $V_T$) \cite[Lemma 6.11]{MS}; let us set $M_{i,T}:=M_{i,r_T}$.
\begin{thm}[{\cite[Theorem 6.12]{MS}}]\label{thm:coxeter_action}
	Let $\mu\in\yn(\irrG)$. Consider the basis $\underset{T\in\tab_G(n,\mu)}{\cup}\B_T$ of $V^{\mu}$. Recall that $(i,i+1)_G:=(\Gid,\cdots,\Gid;(i,i+1)),1\leq i\leq n-1$, and $I_T$ denote the identity matrix of order $\dim(V_T)$. Then, the action of $(i,i+1)_G$ on $V_T,\;T\in\tab_{G}(n,\mu)$ is given below.
	\begin{enumerate}
		\item Set $S=(i,i+1)\cdot T$, the Young $G$-tableau obtained from $T$ by interchanging $i$ and $i+1$. If $i$ and $i+1$ are not in the same Young diagram of $T$, then $V_T\oplus V_S$ is closed under the action of $(i,i+1)_G$ and the matrix of this action with respect to the basis $\B_T\cup\B_S$ is given as follows
		\[\big[(i,i+1)_G\big]_{\B_T\cup\B_S}:=\begin{pmatrix}
			0&I_T\\
			I_T&0
		\end{pmatrix}\]
	\item If $i$ and $i+1$ are in the same row (respectively, column) of the same Young diagram of $T$, then $V_T$ is closed under the action of $(i,i+1)_G$ and the matrix of this action with respect to the basis $\B_T$ is $M_{i,T}$ (respectively, $-M_{i,T}$).
	\item Suppose $i$ and $i+1$ are in the same Young diagram of $T$, but not in the same row or same column of this Young diagram. Set $S=(i,i+1)\cdot T$; therefore $M_{i,S}=M_{i,r_S}$, $r_S=(r_S(1),\dots,r_S(n))$. Then, $V_T\oplus V_S$ is closed under the action of $(i,i+1)_G$ and the matrix of this action with respect to the basis $\B_T\cup\B_S$ is given as follows
	\[\big[(i,i+1)_G\big]_{\B_T\cup\B_S}:=\begin{pmatrix}
		r^{-1}M_{i,S}&\sqrt{1-r^{-2}}\;I_T\\& &\\
		\sqrt{1-r^{-2}}\;I_T&-r^{-1}M_{i,S}
	\end{pmatrix},\]
	where $r=c(b_T(i+1))-c(b_T(i))$. See Definition \ref{def:b_(i)+r_T(i)} for the notations.
	\end{enumerate}
\end{thm}
Let us set some notations that we follow from now on. Let $t:=|\irrG|$ and $\irrG:=\{\sigma_1,\dots,\sigma_t\}$, where $\sigma_1=\trivial$ (the trivial representation of $G$). We write $\mu\left(\in\ynirr\right)$ as the tuple $(\mu^{(1)},\dots,\mu^{(t)})$, where $\mu^{(i)}:=\mu(\sigma_i)$ for each $1\leq i\leq t$. We also denote $m_i:=|\mu^{(i)}|,\;W^{\sigma_i}$ is the irreducible $G$-module corresponding to $\sigma_i$ and $d_i=\dim(W^{\sigma_i})$ for each $1\leq i\leq t$. For $T\in\tab_{G}(n,\mu)$, the dimension of $V_T$ is $d_1^{m_1}\cdots d_t^{m_t}$. For $1\leq i<j\leq t$, we set
\begin{align}\label{eq:notations-YGD}
		\MYGD_{i,j}:=&
		\begin{split}
			&\hspace{1.75cm}(n-1)\text{ boxes}\\[-2.5ex]
			&\tiny{\left(\emptyset,\dots,\emptyset,\underset{\uparrow}{\begin{array}{c}\overbrace{\young(\;\;{\;$\cdots$\;\;}\;\;)}\end{array}},\emptyset,\dots,\emptyset,\underset{\uparrow}{\begin{array}{c}\yng(1)\end{array}},\emptyset,\dots,\emptyset\right)} \in\ynirr.\\[-1ex]
			&\hspace{1.75cm}i\text{th position.}\hspace{1.1cm}j\text{th position.}
		\end{split}
\end{align}
Also, for $k\in[n]$, we set 
\begin{align}\label{eq:notations-YGT}
	\T_{i,j;k}:=&
	\begin{split}
		&\hspace{1.75cm}(n-1)\text{ boxes}\\[-2.5ex]
		&\tiny{\left(\emptyset,\dots,\emptyset,\underset{\uparrow}{\begin{array}{c}\overbrace{\young(1{\;$\cdots\xcancel{k}\cdots$\;\;}n)}\end{array}},\emptyset,\dots,\emptyset,\begin{array}{c}\young(k)\end{array},\emptyset,\dots,\emptyset\right)} \in\tab_G(n,\MYGD_{i,j}).\\[-1ex]
		&\hspace{1.75cm}k\text{ is not present.}
	\end{split}
\end{align}
For $2\leq j\leq t$, we simply denote $\MYGD_{1,j}$ by $\MYGD_j$, and $\T_{1,j;k}$ by $\T_{j;k}$.
\begin{prop}\label{prop:permutation-matrix}
	Given $\pi\in S_n$, recall that $\pi_G:=(\Gid,\dots,\Gid;\pi)\in G\wr S_n$. Let us consider the ordered basis $\B_j:=\B_{\T_{j;1}}\cup\dots\cup\B_{\T_{j;n}}$ of the irreducible $G\wr S_n$-module $V^{\MYGD_j}$, obtained by first listing down the elements of $\B_{\T_{j;1}}$, then the elements of $\B_{\T_{j;2}}$, and so on up to the elements of $\B_{\T_{j;n}}$. Then, the matrix of the action of $\pi_G$ on $V^{\MYGD_j}$ with respect to the basis $\B_j$ is given by
	\[\big[\pi_G\big]_{\B_j}=\big[\pi\big]_{n}\otimes I_{d_j}.\]
	Here, $\big[\pi\big]_{n}$ is the permutation matrix of size $n\times n$, $\otimes$ denotes the tensor (alternatively, Kronecker) product of matrices, and $I_{d_j}$ is the identity matrix of size $d_j:=\dim(W^{\sigma_j})$.
\end{prop}
\begin{proof}
	Let us choose $x\in[n]$. We first show the following:
	\begin{equation}\label{eq:T-M0}
		\big[(x,x+1)_G\big]_{\B_j}:=\big[(x,x+1)\big]_{n}\otimes I_{d_j}.
	\end{equation}
	For all $y\in[n]\setminus\{x,x+1\}$, $\T_{j;y}\in\tab_G(n,\MYGD_j)$, recall $r_{\T_{j;y}}(\bullet)$ from Definition \ref{def:b_(i)+r_T(i)}, and
	\begin{align*}
		r_{\T_{j;y}}&:=(r_{\T_{j;y}}(1),\dots,r_{\T_{j;y}}(y-1),r_{\T_{j;y}}(y),r_{\T_{j;y}}(y+1),\dots,r_{\T_{j;y}}(n))\\
		&\;=(\trivial,\dots,\trivial,\sigma_j,\trivial,\dots,\trivial).
	\end{align*}
	Therefore, the matrix $N_{x,\T_{j;y}}$ in the second part of Theorem \ref{thm:coxeter_action} is $I_{d_j}$. Hence the second part of Theorem \ref{thm:coxeter_action} implies that
	\begin{equation}\label{eq:T-M1}
		\big[(x,x+1)_G\big]_{\B_{\T_{j;y}}}=I_{d_j}.
	\end{equation}
	Again, the first part of Theorem \ref{thm:coxeter_action} implies that
	\begin{equation}\label{eq:T-M2}
		\big[(x,x+1)_G\big]_{\B_{\T_{j;x}}\cup\B_{\T_{j;x+1}}}=\begin{pmatrix}
			0&	I_{d_j}\\
			I_{d_j}&0
		\end{pmatrix}.
	\end{equation}
	Therefore, using \eqref{eq:T-M1} and \eqref{eq:T-M2}, the matrix of the action of $(x,x+1)_G$ on $V^{\MYGD_j}$ with respect to the ordered basis $\B_j:=\B_{\T_{j;1}}\cup\dots\cup\B_{\T_{j;n}}$ is given by \eqref{eq:T-M0}.
	
	As the elementary transpositions of the form $(k,k+1)$ generate $S_n$, we can write $\pi$ as a product of elementary transpositions. Suppose $\pi=\transp_1\cdots \transp_{\ell(\pi)}$, where $\transp_1,\cdots, \transp_{\ell(\pi)}$ are elementary transpositions of the form $(k,k+1)$; then, $\pi_G=(\transp_1)_G\cdots (\transp_{\ell(\pi)})_G$. Therefore, using  \eqref{eq:T-M0}, and the definition of a representation, we have
	\begin{align*}
		\big[\pi_G\big]_{\B_j}=\big[(\transp_1)_G\big]_{\B_j}\cdots \big[(\transp_{\ell(\pi)})_G\big]_{\B_j}&=\left(\big[\transp_1\big]_n\otimes I_{d_j}\right)\cdots \left(\big[\transp_{\ell(\pi)}\big]_n\otimes I_{d_j}\right)\\
		&=\big[\transp_1\big]_n\cdots \big[\transp_{\ell(\pi)}\big]_n\otimes I_{d_j}\\
		&=\big[\transp_1\cdots \transp_{\ell(\pi)}\big]_n\otimes I_{d_j}=\big[\pi\big]_n\otimes I_{d_j}.\qedhere
	\end{align*}
\end{proof}
\begin{prop}\label{prop:def-action}
	Let $(g_1,\dots,g_n;\pi)\in G\wr S_n$. Consider the $G\wr S_n$-module $\MYGD:=V^{\MYGD_0}\oplus V^{\MYGD_1}$,
	\begin{align}\label{eq:def-action-not}
		\begin{split}
			&\hspace{0.75cm}n\emph{ boxes}\\[-1ex]
			\MYGD_0:=&\tiny{\left(\begin{array}{c}\overbrace{\young(\;\;{\;$\cdots$\;\;}\;\;)}\end{array},\emptyset,\dots,\dots,\emptyset\right)},
		\end{split}&\begin{split}
			&\hspace{0.25cm}(n-1)\emph{ boxes}\\[-1ex]
			\MYGD_1:=&\tiny{\left(\begin{array}{c}\overbrace{\young(\;\;{\;$\cdots$\;\;}\;\;,\;)}\end{array},\emptyset,\dots,\dots,\emptyset\right)} \in\ynirr.
		\end{split}
	\end{align}
	Then, the matrix of the action of $(g_1,\dots,g_n;\pi)$ on $\MYGD$ with respect to some basis of $\MYGD$ is the permutation matrix $\big[\pi\big]_n$.
\end{prop}
\begin{proof}
	To prove this lemma, we appeal to the representation theory of the symmetric group $S_n$. Recall that the irreducible representations of $S_n$ are indexed by the partitions of $n$. The irreducible $S_n$-module indexed by the partition $\eta\vdash n$, denoted $S^{\eta}$, known as the \emph{Specht module} indexed by $\eta$. Let us set $\MYGD_{\eta}:=\left(\eta,\emptyset,\dots,\emptyset\right)\in\ynirr$. Then Theorem \ref{thm:GT-decomposition} implies that the action of $(g_1,\dots,g_n;\pi)=(\Gid,\dots,\Gid;\pi)\cdot(g_{\pi(1)},\dots,g_{\pi(n)};\1)$ on $V^{\MYGD_{\eta}}$ is the same as that of $\pi_G:=(\Gid,\dots,\Gid;\pi)$ on $V^{\MYGD_{\eta}}$. We note that the Gelfand-Tsetlin subspaces of $V^{\MYGD_{\eta}}$ are one-dimensional, known as the Gelfand-Tsetlin vectors. The branching rule for $G\wr S_n$ (see Theorem \ref{thm:Branching-rule}) when applied to $V^{\MYGD_{\eta}}$ is the same as the branching rule of $S_n$ \cite[Theorem 2.8.3]{Sagan} applying to the Specht module $S^{\eta}$, and hence the Gelfand-Tsetlin vectors of $V^{\MYGD_{\eta}}$ match with the Gelfand-Tsetlin vectors of $S^{\eta}$. For $T\in\tab_{G}(n,\MYGD_{\eta})$, recall the orthonormal basis $\B_T$ (consisting of a single unit vector), $S:=(i,i+1)\cdot T$ and the matrix $M_{i,S}$ for $1\leq i<n$, and $I_T$ from Theorem \ref{thm:coxeter_action}. Then, for $1\leq i<n$, the actions of $(i,i+1)_G$ on $V^{\MYGD_{\eta}}$ are the same as that of the corresponding simple transpositions $(i,i+1)$ on the Specht module $S^{\eta}$ with respect to the Gelfand-Tsetlin basis (the basis formed by the Gelfand-Tsetlin vectors); thanks to Theorem \ref{thm:coxeter_action}, $M_{i,T}=M_{i,S}=I_T=[1]_{1\times 1}$,  
	and \cite[p. 600, eq (7.5)]{VO}. Therefore, the action of $\pi_G$ on $V^{\MYGD_{\eta}}$ (and hence that of $(g_1,\dots,g_n;\pi)$ on $V^{\MYGD_{\eta}}$) is the same as the action of the permutation $\pi$ on the Specht module $S^{\eta}$. Thus, the proposition follows from the fact that the natural action of $S_n$ on the set $\{1,\dots,n\}$ decomposes as $S^{(n)}\oplus S^{(n-1,1)}$, i.e.,
	\[S^{\tiny{\young(\;\;{\;$\cdots$\;\;}\;)}}\oplus S^{\tiny{\young(\;\;{\;$\cdots$\;\;}\;,\;)}}.\qedhere\]
\end{proof}
\begin{lem}\label{lem:character-irrg}
	Let $(g_1,\dots,g_n;\pi)\in G\wr S_n$. Recall $\MYGD_j:=\MYGD_{1,j}\in\ynirr,\;2\leq j\leq t$, from \eqref{eq:notations-YGD}. Then the irreducible character $\chi^{\MYGD_j}$ of $G\wr S_n$ indexed by the Young $G$-diagram $\MYGD_j$ satisfies	\[\chi^{\MYGD_j}(g_1,\dots,g_n;\pi)=\sum_{i=1}^{n}\delta_{i,\pi(i)}\chi^j(g_i)\text{ for }2\leq j\leq t.\]
	Here, $\chi^j$ denotes the irreducible character of $G$ indexed by $\sigma_j$.
\end{lem}
\begin{proof}
For $\MYGD_j:=\MYGD_{1,j}\in\ynirr$, recall $\T_{j;k}:=\T_{1,j;k}\in\tab_G(n,\MYGD_j)$ from \eqref{eq:notations-YGT}. Also recall
\[r_{\T_{j;k}}=\left(r_{\T_{j;k}}(1),\dots,r_{\T_{j;k}}(k),\cdots,r_{\T_{j;k}}(n)\right)=(\trivial,\dots,\sigma_j,\dots,\trivial),\]
and the ordered basis $\B_j$ of the irreducible $G\wr S_n$-module $V^{\MYGD_j}$ from Proposition \ref{prop:permutation-matrix}. Now, using Theorem \ref{thm:GT-decomposition}, the action of $(g_1,\dots,g_n;\1)\in G\wr S_n$ on the Gelfand-Tsetlin subspace $V_{\T_{j;k}}$ with respect to the basis $\B_{\T_{j;k}}$ is given by the matrix $\big[g_k\big]_{B^{\sigma_j}}$, where $B^{\sigma_j}$ is the (orthonormal) basis of $W^{\sigma_j}$. Therefore, the matrix of the action of $(g_1,\dots,g_n;\1)$ on $V^{\MYGD_j}$ with respect to the basis $\B_j$ is given as follows:
\begin{equation}\label{eq:C-I1}
	\big[(g_1,\dots,g_n;\1)\big]_{\B_j}=\underset{1\leq k\leq n}{\oplus}\big[g_k\big]_{B^{\sigma_j}}:=
	\begin{pmatrix}
		[g_1]_{B^{\sigma_j}}& & &  &\\
		&[g_2]_{B^{\sigma_j}}& &  &\\
		& &\ddots&  &\\
		& & & [g_{n-1}]_{B^{\sigma_j}}  &\\
		& & & & [g_n]_{B^{\sigma_j}}
	\end{pmatrix}.
\end{equation}
Again, recalling $d_j=\dim(W^{\sigma_j})$, Proposition \ref{prop:permutation-matrix} implies that
\begin{equation}\label{eq:C-I2}
	\big[\pi_G\big]_{\B_j}=\big[\pi\big]_n\otimes I_{d_j}=
	\begin{pmatrix}
		\delta_{1,\pi(1)}I_{d_j}&\delta_{1,\pi(2)}I_{d_j}&\delta_{1,\pi(3)}I_{d_j}&\cdots&\cdots&\delta_{1,\pi(n)}I_{d_j}\\
		& & & & &\\
		\delta_{2,\pi(1)}I_{d_j}&\delta_{2,\pi(2)}I_{d_j}&\delta_{2,\pi(3)}I_{d_j}&\cdots&\cdots&\delta_{2,\pi(n)}I_{d_j}\\
		\vdots&\vdots&\vdots&\vdots&\vdots&\vdots\\
		& & &\delta_{p,\pi(q)}I_{d_j}& & \\
		\vdots&\vdots&\vdots&\vdots&\vdots&\vdots\\
		\delta_{n,\pi(1)}I_{d_j}&\delta_{n,\pi(2)}I_{d_j}&\delta_{n,\pi(3)}I_{d_j}&\cdots&\cdots&\delta_{n,\pi(n)}I_{d_j}
	\end{pmatrix}
\end{equation}
Now using the definition of a character, we have
\begin{align*}
	\chi^{\MYGD_j}(g_1,\dots,g_n;\pi)
	&=\Tr\left(\big[(g_1,\dots,g_n;\1)\cdot (\Gid,\dots,\Gid;\pi)\big]_{\B_j}\right)\\
	&=\Tr\left(\big[(g_1,\dots,g_n;\1)\big]_{\B_j} \big[\pi_G\big]_{\B_j}\right)\\
	&=\sum_{i=1}^{n}\Tr\left(\delta_{i,\pi(i)}[g_i]_{B^{\sigma_j}}\right)\\
	&=\sum_{i=1}^{n}\delta_{i,\pi(i)}\Tr\left([g_i]_{B^{\sigma_j}}\right)=\sum_{i=1}^{n}\delta_{i,\pi(i)}\chi^j(g_i).\qedhere
\end{align*}
\end{proof}
We now formally define a representation $\mathcal{R}_n:G\wr S_n\rightarrow \text{GL}\left(\mathbb{C}[G\times [n]]\right)$ on the basis elements of $\mathbb{C}[G\times [n]]$ by the action in \eqref{eq:grp_action}, i.e.,
\begin{equation}\label{eq:grp_action-repn}
	\mathcal{R}_n:(g_1,\dots,g_n;\pi)\mapsto\left(\sum_{(h,i)}a_{(h,i)}(h,i)\mapsto\sum_{(h,i)}a_{(h,i)}(g_{\pi(i)}h,\pi(i))\right),\text{ where }a_{(h,i)}\in\mathbb{C}.
\end{equation}
To decompose $\mathbb{C}[G\times [n]]$ into irreducible $G\wr S_n$-modules, we first prove the following lemma.
\begin{lem}\label{lem:character-natural_action}
	Recall that $L$ denotes the left regular representation of $G$. Let $\chi^{\mathcal{R}_n}$ (respectively $\chi^L$) be the character of the representation $\mathcal{R}_n$ (respectively $L$). Then,
	\[\chi^{\mathcal{R}_n}(g_1,\dots,g_n;\pi)=\sum_{s=1}^{n}\delta_{s,\pi(s)}\chi^L(g_{\pi(s)}).\]
\end{lem}
\begin{proof}
	Let us fix the following ordering in the basis $G\times [n]$: First, list down the elements of $G\times\{1\}$, then the elements of $G\times\{2\}$, and so on up to the elements of $G\times\{n\}$. Then the matrix of the action of $(g_1,\dots,g_n;\pi)$ on $\mathbb{C}[G\times [n]]$ with respect to the ordered basis $G\times[n]=\left(G\times\{1\}\right)\cup \left(G\times\{2\}\right)\cup\cdots\cup \left(G\times\{n\}\right)$ is given as follows:
	\begin{align}\label{eq:C-Naction1}
		&\;\big[\mathcal{R}_n(g_1,\dots,g_n;\pi)\big]_{G\times[n]}\\
		=&\begin{pmatrix}
			\delta_{1,\pi(1)}L(g_{\pi(1)})&\delta_{1,\pi(2)}L(g_{\pi(2)})&\delta_{1,\pi(3)}L(g_{\pi(3)})&\cdots&\cdots&\delta_{1,\pi(n)}L(g_{\pi(n)})\\
			& & & & &\\
			\delta_{2,\pi(1)}L(g_{\pi(1)})&\delta_{2,\pi(2)}L(g_{\pi(2)})&\delta_{2,\pi(3)}L(g_{\pi(3)})&\cdots&\cdots&\delta_{2,\pi(n)}L(g_{\pi(n)})\\
			\vdots&\vdots&\vdots&\vdots&\vdots\\
			&  & &\delta_{p,\pi(q)}L(g_{\pi(q)})& & \\
			\vdots&\vdots&\vdots&\vdots&\vdots&\vdots\\
			\delta_{n,\pi(1)}L(g_{\pi(1)})&\delta_{n,\pi(2)}L(g_{\pi(2)})&\delta_{n,\pi(3)}L(g_{\pi(3)})&\cdots&\cdots&\delta_{n,\pi(n)}L(g_{\pi(n)})\nonumber
		\end{pmatrix}
	\end{align}
The lemma follows by taking trace on both sides of \eqref{eq:C-Naction1}.
\end{proof}
\begin{thm}\label{thm:grp_action-decomposition}
	Recall $\irrG:=\{\sigma_1,\dots,\sigma_t\},\sigma_{1}=\trivial$, $\MYGD_{0}$ and $\MYGD_{1}$ from \eqref{eq:def-action-not}, and $\MYGD_j=\MYGD_{1,j}\in\ynirr,2\leq j\leq t$ from \eqref{eq:notations-YGD}. Then, the decomposition of $\mathbb{C}[G\times [n]]$ into irreducible $G\wr S_n$-modules is given by
	\[\mathbb{C}[G\times [n]]\cong\underset{0\leq j\leq t}{\oplus}d_j\;V^{\MYGD_j},\text{ where }d_0=d_1,\]
	and $d_j$ is the dimension of the irreducible representation of $G$ indexed by $\sigma_j$ for $1\leq j\leq t$.
\end{thm}
\begin{proof}
	Let us first recall $\chi^{\mathcal{R}_n}$ is the character of the representation $\mathcal{R}_n$, and $\chi^{\MYGD_j},\;0\leq j\leq t$, is the irreducible character of $G\wr S_n$ indexed by $\MYGD_j$. Then it is enough to prove the following
	\begin{equation}\label{eq:g_a-d1}
		\chi^{\mathcal{R}_n}=\displaystyle\sum_{j=0}^{t}d_j\chi^{\MYGD_j}.
	\end{equation}
	 Theorem \ref{thm:ch2_irreducible_character_as_ONB} implies that \eqref{eq:g_a-d1} is equivalent to show $\langle\chi^{\mathcal{R}_n},\chi^{\MYGD_j}\rangle=d_j$ for all $0\leq j\leq t$. We also recall the notation $\chi^j$ for the irreducible character of $G$ indexed by $\sigma_j$ and the fact that the left regular representation of a group decomposes into irreducible representations with multiplicity equal to their respective dimensions (see \eqref{eq:Group_alg._decom.}). Now, using Lemma \ref{lem:character-natural_action} and the fact that $\chi^{\MYGD_0}$ is the trivial character of $G\wr S_n$, we have the following:
	 \begin{align}\label{eq:g_a-d1.0}
	 	\langle\chi^{\mathcal{R}_n},\chi^{\MYGD_0}\rangle&=\frac{1}{|G|^nn!}\sum_{(g_1,\dots,g_n;\pi)\in G\wr S_n}\chi^{\mathcal{R}_n}(g_1,\dots,g_n;\pi)\overline{\chi^{\MYGD_0}(g_1,\dots,g_n;\pi)}\nonumber\\
	 	&=\frac{1}{|G|^nn!}\sum_{(g_1,\dots,g_n;\pi)\in G\wr S_n}\sum_{s=1}^{n}\delta_{s,\pi(s)}\chi^L(g_{\pi(s)})\cdot 1\nonumber\\
	 	&=\frac{1}{|G|^nn!}\sum_{s=1}^{n}\sum_{(g_1,\dots,g_n;\pi)\in G\wr S_n}\delta_{s,\pi(s)}\chi^L(g_{\pi(s)})\nonumber\\
	 	&=\frac{1}{|G|^nn!}\sum_{s=1}^{n}|G|^{n-1}(n-1)!\sum_{g\in G}\chi^L(g)\nonumber\\
	 	&=\frac{1}{|G|}\sum_{g\in G}\chi^L(g)\overline{\chi^1(g)}=\langle\chi^L,\chi^1\rangle=1.
	 \end{align}
 	Now Proposition \ref{prop:def-action} implies that $\chi^{\MYGD_1}=\chi^{\MYGD}-\chi^{\MYGD_0}$, where $\chi^{\MYGD}$ is the character of the representation of $G\wr S_n$ indexed by $\MYGD$. Now, using Lemma \ref{lem:character-natural_action} and Proposition \ref{prop:def-action}, we have,
 	 \begin{align}\label{eq:g_a-d1.1}
 		\langle\chi^{\mathcal{R}_n},\chi^{\MYGD_1}\rangle+\langle\chi^{\mathcal{R}_n},\chi^{\MYGD_0}\rangle&=\langle\chi^{\mathcal{R}_n},\chi^{\MYGD}\rangle\nonumber\\
 		&=\frac{1}{|G|^nn!}\sum_{(g_1,\dots,g_n;\pi)\in G\wr S_n}\chi^{\mathcal{R}_n}(g_1,\dots,g_n;\pi)\overline{\chi^{\MYGD}(g_1,\dots,g_n;\pi)}\nonumber\\
 		&=\frac{1}{|G|^nn!}\sum_{(g_1,\dots,g_n;\pi)\in G\wr S_n}\sum_{s=1}^{n}\delta_{s,\pi(s)}\chi^L(g_{\pi(s)})\cdot\sum_{k=1}^{n}\delta_{k,\pi(k)}\nonumber\\
 		&=\frac{1}{|G|^nn!}\sum_{s=1}^{n}\sum_{k=1}^{n}\sum_{(g_1,\dots,g_n;\pi)\in G\wr S_n}\delta_{k,\pi(k)}\delta_{s,\pi(s)}\chi^L(g_{\pi(s)})\nonumber\\
 		&=\frac{1}{|G|^nn!}\sum_{s,k=1}^{n}|G|^{n-1}\left((n-1)!\delta_{s,k}+(n-2)!(1-\delta_{s,k})\right)\sum_{g\in G}\chi^L(g)\nonumber\\
 		&=\frac{2}{|G|}\sum_{g\in G}\chi^L(g)\overline{\chi^1(g)}=2\;\langle\chi^L,\chi^1\rangle=2.
 	\end{align}
 	Now let $2\leq j\leq t$. Then Lemma \ref{lem:character-irrg} and Lemma \ref{lem:character-natural_action} implies the following:
 	\begin{align}\label{eq:g_a-d1.2}
 		\langle\chi^{\mathcal{R}_n},\chi^{\MYGD_j}\rangle&=\frac{1}{|G|^nn!}\sum_{(g_1,\dots,g_n;\pi)\in G\wr S_n}\chi^{\mathcal{R}_n}(g_1,\dots,g_n;\pi)\overline{\chi^{\MYGD_j}(g_1,\dots,g_n;\pi)}\nonumber\\
 		&=\frac{1}{|G|^nn!}\sum_{(g_1,\dots,g_n;\pi)\in G\wr S_n}\sum_{s=1}^{n}\delta_{s,\pi(s)}\chi^L(g_{\pi(s)})\cdot\overline{\sum_{k=1}^{n}\delta_{k,\pi(k)}\chi^j(g_k)}\nonumber\\
 		&=\frac{1}{|G|^nn!}\sum_{s=1}^{n}\sum_{k=1}^{n}\sum_{(g_1,\dots,g_n;\pi)\in G\wr S_n}\delta_{k,\pi(k)}\delta_{s,\pi(s)}\chi^L(g_{\pi(s)})\overline{\chi^j(g_k)}\nonumber\\
 		&=\frac{1}{|G|^nn!}\sum_{s,k=1}^{n}|G|^{n-1}\delta_{s,k}(n-1)!\sum_{g\in G}\chi^L(g)\overline{\chi^j(g)}\nonumber\\
 		&\hspace*{1in}+\frac{1}{|G|^nn!}\sum_{s,k=1}^{n}|G|^{n-2}(n-2)!(1-\delta_{s,k})\sum_{g,g'\in G}\chi^L(g)\overline{\chi^j(g')}\nonumber\\
 		&=\frac{1}{|G|}\sum_{g\in G}\chi^L(g)\overline{\chi^j(g)}+\frac{1}{|G|^2}\left(\sum_{g\in G}\chi^L(g)\overline{\chi^1(g)}\right)\left(\sum_{g'\in G}\chi^1(g')\overline{\chi^j(g')}\right)\nonumber\\
 		&=\langle\chi^L,\chi^j\rangle+\langle\chi^L,\chi^1\rangle\langle\chi^1,\chi^j\rangle=d_j
 	\end{align}
 	The last equality in \eqref{eq:g_a-d1.2} follows from the fact that $j\neq 1$. As \eqref{eq:g_a-d1.0} and \eqref{eq:g_a-d1.1} implies $\langle\chi^{\mathcal{R}_n},\chi^{\MYGD_0}\rangle=1=d_0$ and $\langle\chi^{\mathcal{R}_n},\chi^{\MYGD_1}\rangle=1=d_1$, the proof follows from \eqref{eq:g_a-d1.2}.
\end{proof}
\section{Proof of Theorem \ref{thm:main_thm}}\label{sec:main_proof}
The group algebra element $\mathcal{A}=\sum_{g\in G}\alpha_gg$ is symmetric and non-negative (see \eqref{eq:group_alg_sett}); therefore, the matrix of the action of $\mathcal{A}$ on $\mathbb{C}[G]$ by multiplication on the right is self-adjoint. Let us recall the $G$-invariant inner product that appeared in the discussion after Theorem \ref{thm:dimen of irr G_n-modules}. For each $\sigma\in\irrG$, we choose the orthonormal basis $B^{\sigma}=\{v^{\sigma}_1,\dots,v^{\sigma}_{\dim(W^{\sigma})}\}$ of the irreducible $G$-module $W^{\sigma}$ such that the matrix of the action of $\mathcal{A}\in\mathbb{R}_+[G]^{(\s)}$ (see \eqref{eq:Re-Sym_Gp-Alg}) on $W^{\sigma}$, with respect to basis $B^{\sigma}$, is a diagonal matrix, thanks to the spectral theorem for the self-adjoint operators. Let $\Delta(\mathcal{A}):=\sum_{g\in G}\alpha_g\left(\Gid-g\right)$ and
\begin{equation}\label{eq:diagonal_matrix}
	\begin{split}
	&\Delta(\mathcal{A})\cdot v^{\sigma}_i=\displaystyle\sum_{g\in G}\alpha_g\left(\Gid-g\right)\cdot v^{\sigma}_i=D^{\sigma}_i\cdot v^{\sigma}_i,\text{ for }1\leq i\leq \dim(W^{\sigma}),\\
	&\text{ i.e., }\quad\Delta_{G}(\mathcal{A},\sigma)=\D(\sigma):=\diag\left(D^{\sigma}_1,\dots,D^{\sigma}_{\dim(W^{\sigma})}\right)
	\end{split}
\end{equation} 
be the diagonal matrix of the action of $\Delta(\mathcal{A})$ on $W^{\sigma}$ with respect to the basis $B^{\sigma}$. Moreover,
\begin{equation*}
	\begin{cases}
		\D(\trivial):=\big[D_1^{\trivial}\big]_{1\times 1}=[0]_{1\times 1},\\
		D^{\sigma}_i>0\text{ for all }1\leq i\leq\dim(W^{\sigma})&\text{ if }\sigma\in\irrG\setminus\{\trivial\},
	\end{cases}
\end{equation*}
as the assumption $\alpha_g>0\;(g\in G)$ implies that the underlying continuous-time random walk on $G$ is irreducible. We will accordingly choose the basis $\B_T$ for the Gelfand-Tsetlin subspaces $V_T$ of the irreducible $G\wr S_n$-module $V^{\mu}$, where $\mu\in\ynirr$ and $T\in\tab_G(n,\mu)$. From now on, we will work with the basis $\B_T,\;T\in\tab_G(n,\mu)$ and $\mu\in\ynirr$. Given any positive integer $k$, we denote the identity matrix of size $k\times k$ by $I_k$. We prove the main theorem in this section. We use the first principle of mathematical induction on $n$ for the proof. We begin with proving the base case (for $n=2$) of the mathematical induction as follows:
\begin{thm}\label{thm:base-case}
	Recall the representation $\mathcal{R}_n$ of $G\wr S_n$ from \eqref{eq:grp_action-repn} for the particular case when $n=2$. Let us consider the group algebra element
	\[\mathcal{G}_2:=x_{1,2}(1,2)_G+\sum_{w=1}^{2}y_w\sum_{g\in G}\alpha_g\left(g\right)^{(w)}\in\mathbb{C}[G\wr S_2],\]
	$x_{1,2},\;y_1,\;y_2>0$, and $\alpha_g=\alpha_{g^{-1}}\geq 0$ for all $g\in G$, such that
	\[\Supp(\mathcal{G}_2):=\{(1,2)_G,(g)^{(w)}:\alpha_g>0,1\leq w\leq 2,g\in G\}\]
	generates the group $G\wr S_2$. Then, using the notational setup of \eqref{eq:sp-gap-rep} and \eqref{eq:sp-gap-rep1}, we have
	\begin{equation}\label{eq:base}
		\psi_{G\wr S_2}(\mathcal{G}_2)= \psi_{G\wr S_2}(\mathcal{G}_2,\mathcal{R}_2).
	\end{equation}
\end{thm}
\begin{proof}
	Let us recall that $t:=|\irrG|$ and $\irrG:=\{\sigma_1,\dots,\sigma_t\}$, where $\sigma_1=\trivial$, the trivial representation of $G$. We write $\mu\in\y_2(\irrG)$ as the tuple $(\mu^{(1)},\dots,\mu^{(t)})$, where $\mu^{(i)}:=\mu(\sigma_i)$ for each $1\leq i\leq t$. $W^{\sigma_i}$ is the irreducible $G$-module corresponding to $\sigma_i$ and $d_i=\dim(W^{\sigma_i})$ for each $1\leq i\leq t$. Then, the set of all Young $G$-diagrams with $2$ boxes is given as follows:
	\begin{equation}\label{eq:base-case0}
		\y_2(\irrG)=\{\NYGD_{i,j},\NYGD_k,\NYGD_k':1\leq i<j\leq t,1\leq k\leq t\},\text{ where }
	\end{equation}
	\begin{align*}
		&\begin{split}
			\NYGD_{i,j}:=&\tiny{\left(\emptyset,\dots,\emptyset,\underset{\uparrow}{\begin{array}{c}\yng(1)\end{array}},\emptyset,\dots,\emptyset,\underset{\uparrow}{\begin{array}{c}\yng(1)\end{array}},\emptyset,\dots,\emptyset\right)},\\[-1ex]
			&\hspace{1cm}i\text{th position.}\hspace{0.5cm}j\text{th position.}
		\end{split}\quad\begin{split}
			\NYGD_{k}:=&\tiny{\left(\emptyset,\dots,\emptyset,\underset{\uparrow}{\begin{array}{c}\yng(2)\end{array}},\emptyset,\dots,\emptyset\right)},\\[-1ex]
			&\hspace{1cm}k\text{th position.}
		\end{split}\\
	&\begin{split}
		\text{ and }\NYGD_{k}':=&\tiny{\left(\emptyset,\dots,\emptyset,\underset{\uparrow}{\begin{array}{c}\yng(1,1)\end{array}},\emptyset,\dots,\emptyset\right)}.\\[-1ex]
		&\hspace{1cm}k\text{th position.}
	\end{split}
	\end{align*}
Let us choose arbitrary integers $i,j,k$ satisfying $1\leq i<j\leq t$ and $1\leq k\leq t$. Now, set the following notations for the standard Young $G$-tableaux with $2$ boxes:
\begin{align*}
	\begin{split}
		\T_{k}:=&\tiny{\left(\emptyset,\dots,\emptyset,\underset{\uparrow}{\begin{array}{c}\young(12)\end{array}},\emptyset,\dots,\emptyset\right)}\in\tab_{G}(2,\NYGD_{k}),\\[-1ex]
		&\hspace{1cm}k\text{th position.}
	\end{split}\;\;\\
	\begin{split}
		\T_{k}':=&\tiny{\left(\emptyset,\dots,\emptyset,\underset{\uparrow}{\begin{array}{c}\young(1,2)\end{array}},\emptyset,\dots,\emptyset\right)}\in\tab_{G}(2,\NYGD_{k}'),\\[-1ex]
		&\hspace{1cm}k\text{th position.}
	\end{split}
\end{align*}
\begin{align*}
	&\begin{split}
		\T^{1,2}_{i,j}:=&\tiny{\left(\emptyset,\dots,\emptyset,\underset{\uparrow}{\begin{array}{c}\young(1)\end{array}},\emptyset,\dots,\emptyset,\underset{\uparrow}{\begin{array}{c}\young(2)\end{array}},\emptyset,\dots,\emptyset\right)}\in\tab_{G}(2,\NYGD_{i,j}),\\[-1ex]
		&\hspace{1cm}i\text{th position.}\hspace{0.5cm}j\text{th position.}\\
		\T^{2,1}_{i,j}:=&\tiny{\left(\emptyset,\dots,\emptyset,\underset{\uparrow}{\begin{array}{c}\young(2)\end{array}},\emptyset,\dots,\emptyset,\underset{\uparrow}{\begin{array}{c}\young(1)\end{array}},\emptyset,\dots,\emptyset\right)}\in\tab_{G}(2,\NYGD_{i,j}),\\[-1ex]
		&\hspace{1cm}i\text{th position.}\hspace{0.5cm}j\text{th position.}
	\end{split}
\end{align*}
The Gelfand-Tsetlin decomposition of the irreducible $G\wr S_2$ modules $V^{\NYGD_{k}}$ (respectively, $V^{\NYGD_{k}'}$) is given by $V^{\NYGD_{k}}=V_{\T_k}$ (respectively, $V^{\NYGD_{k}'}=V_{\T_{k}'}$). Also, the Gelfand-Tsetlin decomposition of the irreducible $G\wr S_2$ modules $V^{\NYGD_{i,j}}$ is given by $V^{\NYGD_{i,j}}=V_{\T^{1,2}_{i,j}}\oplus V_{\T^{2,1}_{i,j}}$.

Recall the basis $B^{\sigma_k}:=\{v^{\sigma_k}_1,\dots,v^{\sigma_k}_{d_k}\}$ of $W^{\sigma_k}$, and $r_{\T_k}(\bullet)$ from Definition \ref{def:b_(i)+r_T(i)}. We first focus on the Gelfand-Tsetlin subspace $V_{\T_k}$ of the irreducible $G\wr S_2$-module $V^{\NYGD_k}$. We have $r_{\T_k}:=(r_{\T_k}(1),r_{\T_k}(2))=(\sigma_k,\sigma_k)$. Recall the basis $B^{r_{\T_k}}:=\{v^{\sigma_k}_p\otimes v^{\sigma_k}_q:1\leq p,q\leq d_k\}$ of $W^{\sigma_k}\times W^{\sigma_k}$, and the matrix $M_{1,\T_k}$ of the switch operator $\tau_{1,r_{\T_k}}$ with respect to basis $B^{r_{\T_k}}$. Now we reorder the basis $B^{r_{\T_k}}$ as follows: \[B^{r_{\T_k}}=B^{(1)}_{r_{\T_k}}\cup \left(\underset{1\leq p<q\leq d_k}{\cup}B^{(p,q)}_{r_{\T_k}}\right),\]
where $B^{(1)}_{r_{\T_k}}=\{v_p^{\sigma_k}\otimes v_p^{\sigma_k}: 1\leq p\leq d_k\}$ and $B^{(p,q)}_{r_{\T_k}}=\{v_p^{\sigma_k}\otimes v_q^{\sigma_k},v_q^{\sigma_k}\otimes v_p^{\sigma_k}\},1\leq p<q\leq d_k$. Then, using Theorem \ref{thm:GT-decomposition}, Theorem \ref{thm:coxeter_action}, and \eqref{eq:diagonal_matrix}, we have
\begin{align*}
	\Delta_{G\wr S_2}(\mathcal{G}_2,\NYGD_{k})&=\big[\Delta_{G\wr S_2}(\mathcal{G}_2,\NYGD_{k})\big]_{B^{(1)}_{r_{\T_k}}}\oplus\left(\underset{1\leq p<q\leq d_k}{\oplus}\big[\Delta_{G\wr S_2}(\mathcal{G}_2,\NYGD_{k})\big]_{B^{(p,q)}_{r_{\T_k}}}\right)\\
	&=(y_1+y_2)\cdot\begin{pmatrix}
		D^{\sigma_k}_1& & & &\\
		&D^{\sigma_k}_2& & & \\
		& &\ddots& &\\
		& & & & D^{\sigma_k}_{d_k} 
	\end{pmatrix}\\
	&\quad\quad\oplus
	\left(\underset{1\leq p<q\leq d_k}{\oplus}\begin{pmatrix}
		x_{1,2}+y_1D^{\sigma_k}_p+y_2D^{\sigma_k}_q&-x_{1,2}\\&\\
		-x_{1,2}&x_{1,2}+y_1D^{\sigma_k}_q+y_2D^{\sigma_k}_p
	\end{pmatrix}\right)
\end{align*}
Thus, the eigenvalues of $\Delta_{G\wr S_2}(\mathcal{G}_2,\NYGD_{k})$ are given by $(y_1+y_2)D^{\sigma_k}_p,\;1\leq p\leq d_k$, and
\[\frac{1}{2}(y_1+y_2)(D^{\sigma_k}_p+D^{\sigma_k}_q)+x_{1,2}\pm\sqrt{\frac{1}{4}(y_1-y_2)^2(D^{\sigma_k}_p-D^{\sigma_k}_q)^2+x_{1,2}^2},\;1\leq p<q\leq d_k.\]
Note that the eigenvalue for the case $k=1$ is $0$. Now, for $1<k\leq t$ and $1\leq p<q\leq d_k$,
\begin{align}\label{eq:base-case1.1}
	&\frac{1}{2}(y_1+y_2)(D^{\sigma_k}_p+D^{\sigma_k}_q)+x_{1,2}-\sqrt{\frac{1}{4}(y_1-y_2)^2(D^{\sigma_k}_p-D^{\sigma_k}_q)^2+x_{1,2}^2}\\
	\geq&\;\frac{1}{2}(y_1+y_2)(D^{\sigma_k}_p+D^{\sigma_k}_q)+x_{1,2}-\left(\frac{1}{2}|y_1-y_2||D^{\sigma_k}_p-D^{\sigma_k}_q|+x_{1,2}\right)\nonumber\\
	=&\;\frac{1}{2}(y_1+y_2)|D^{\sigma_k}_p-D^{\sigma_k}_q|+(y_1+y_2)\min\{D^{\sigma_k}_p,D^{\sigma_k}_q\}-\frac{1}{2}|y_1-y_2||D^{\sigma_k}_p-D^{\sigma_k}_q|\nonumber\\
	=&\;\min\{y_1,y_2\}|D^{\sigma_k}_p-D^{\sigma_k}_q|+(y_1+y_2)\min\{D^{\sigma_k}_p,D^{\sigma_k}_q\}\geq (y_1+y_2)\min\{D^{\sigma_k}_p,D^{\sigma_k}_q\}.\nonumber
\end{align}
Therefore, we have
\begin{equation}\label{eq:base-case-main-1}
	\psi_{G\wr S_2}(\mathcal{G}_2,\NYGD_{k})=(y_1+y_2)\cdot\underset{1\leq p\leq d_k}{\min}(D^{\sigma_k}_p)\text{ for }1<k\leq t.
\end{equation}
We use a similar argument for the case of the irreducible $G\wr S_2$-module indexed by $\NYGD_k'$. As $r_{\T'_k}:=(r_{\T'_k}(1),r_{\T'_k}(2))=(\sigma_k,\sigma_k)$, the basis $B^{r_{\T_k'}}$ of $W^{\sigma_k}\times W^{\sigma_k}$ matches with $B^{r_{\T_k}}$. Thus using the similar reordering of $B^{r_{\T_k}}$, Theorem \ref{thm:GT-decomposition}, Theorem \ref{thm:coxeter_action}, and \eqref{eq:diagonal_matrix}, we have
\begin{align*}
	\Delta_{G\wr S_2}(\mathcal{G}_2,\NYGD'_{k})
	&=\left(2x_{1,2}\cdot I_{d_k}+(y_1+y_2)\cdot\begin{pmatrix}
		D^{\sigma_k}_1& & & &\\
		&D^{\sigma_k}_2& & & \\
		& &\ddots& &\\
		& & & & D^{\sigma_k}_{d_k} 
	\end{pmatrix}\right)\\
	&\quad\quad\quad\oplus
	\left(\underset{1\leq p<q\leq d_k}{\oplus}\begin{pmatrix}
		x_{1,2}+y_1D^{\sigma_k}_p+y_2D^{\sigma_k}_q&x_{1,2}\\&\\
		x_{1,2}&x_{1,2}+y_1D^{\sigma_k}_q+y_2D^{\sigma_k}_p
	\end{pmatrix}\right)
\end{align*}
Thus the eigenvalues of $\Delta_{G\wr S_2}(\mathcal{G}_2,\NYGD_{k}')$ are given by $2x_{1,2}+(y_1+y_2)D^{\sigma_k}_p,\;1\leq p\leq d_k$, and
\[\frac{1}{2}(y_1+y_2)(D^{\sigma_k}_p+D^{\sigma_k}_q)+x_{1,2}\pm\sqrt{\frac{1}{4}(y_1-y_2)^2(D^{\sigma_k}_p-D^{\sigma_k}_q)^2+x_{1,2}^2},\;1\leq p<q\leq d_k.\]
The eigenvalue for the case $k=1$ is $2x_{1,2}$. For $1<k\leq t$, using \eqref{eq:base-case1.1} and \eqref{eq:base-case-main-1}, we have
\begin{equation}\label{eq:base-case-main-2}
	\psi_{G\wr S_2}(\mathcal{G}_2,\NYGD_{k}')\geq \psi_{G\wr S_2}(\mathcal{G}_2,\NYGD_{k})\text{ for }1<k\leq t.
\end{equation}

We now focus on $V^{\NYGD_{i,j}}$; recall the Gelfand-Tsetlin decomposition $V^{\NYGD_{i,j}}=V_{\T^{1,2}_{i,j}}\oplus V_{\T^{2,1}_{i,j}}$. Then, using Theorem \ref{thm:GT-decomposition}, Theorem \ref{thm:coxeter_action}, and \eqref{eq:diagonal_matrix}, we can write $\Delta_{G\wr S_2}(\mathcal{G}_2,\NYGD_{i,j})$ as follows
\begin{align}\label{eq:base-case1.2}
	&\begin{pmatrix}
		x_{1,2}I_{d_i}\otimes I_{d_j}&-x_{1,2}I_{d_j}\otimes I_{d_i}\\&\\
		-x_{1,2}I_{d_i}\otimes I_{d_j}&x_{1,2}I_{d_j}\otimes I_{d_i}
		\end{pmatrix}
	+\begin{pmatrix}
			y_1\D(\sigma_i)\otimes I_{d_j}&0\\&\\
			0&y_1\D(\sigma_j)\otimes I_{d_i}
		\end{pmatrix}\\
	&\hspace{5cm}+\begin{pmatrix}
		y_2I_{d_i}\otimes \D(\sigma_j)&0\\&\\
		0&y_2I_{d_j}\otimes \D(\sigma_i)
	\end{pmatrix}\nonumber
\end{align}
Therefore, using a proper basis reordering, the matrix \eqref{eq:base-case1.2} become
\begin{align*}
	\Delta_{G\wr S_2}(\mathcal{G}_2,\NYGD_{i,j})=&\underset{\substack{1\leq p\leq d_i\\1\leq q\leq d_j}}{\oplus}\begin{pmatrix}
		x_{1,2}+y_1D^{\sigma_i}_p+y_2D^{\sigma_j}_q&-x_{1,2}\\&\\
		-x_{1,2}&x_{1,2}+y_1D^{\sigma_i}_q+y_2D^{\sigma_j}_p
	\end{pmatrix}
\end{align*}
Thus the eigenvalues of $\Delta_{G\wr S_2}(\mathcal{G}_2,\NYGD_{i,j})$ are listed as follows:
\begin{equation}\label{eq:base-case1.3}
\frac{1}{2}(y_1+y_2)(D^{\sigma_i}_p+D^{\sigma_j}_q)+x_{1,2}\pm\sqrt{\frac{1}{4}(y_1-y_2)^2(D^{\sigma_i}_p-D^{\sigma_j}_q)^2+x_{1,2}^2},\;1\leq p\leq d_i,\;1\leq q\leq d_j.
\end{equation}
We note that for $i=1,d_1=1$ and $D_1^{\sigma_1}=0$. Thus eigenvalues of $\Delta_{G\wr S_2}(\mathcal{G}_2,\NYGD_{1,j})$ become
\begin{equation}\label{eq:base-case1.4}
\frac{1}{2}(y_1+y_2)D^{\sigma_j}_q+x_{1,2}\pm\sqrt{\frac{1}{4}(y_1-y_2)^2\left(D^{\sigma_j}_q\right)^2+x_{1,2}^2},\;\;1\leq q\leq d_j.
\end{equation}
Again, for any two positive real numbers, $A$ and $B$, we have
\begin{align}
	&(y_1+y_2)A\geq\frac{1}{2}(y_1+y_2)A+x_{1,2}-\sqrt{\frac{1}{4}(y_1-y_2)^2\cdot 0+x_{1,2}^2}\label{eq:base-case1.5-1}\\
		&\hspace*{2cm}\geq\frac{1}{2}(y_1+y_2)A+x_{1,2}-\sqrt{\frac{1}{4}(y_1-y_2)^2A^2+x_{1,2}^2}\nonumber\\
	&\frac{1}{2}(y_1+y_2)(A+B)+x_{1,2}-\sqrt{\frac{1}{4}(y_1-y_2)^2(A-B)^2+x_{1,2}^2}\label{eq:base-case1.5-2}\\
	&\quad\geq\frac{1}{2}(y_1+y_2)\min\{A,B\}+x_{1,2}-\sqrt{\frac{1}{4}(y_1-y_2)^2(\min\{A,B\})^2+x_{1,2}^2}\nonumber.
\end{align}
Therefore, from \eqref{eq:base-case1.3}, \eqref{eq:base-case1.4}, and \eqref{eq:base-case1.5-2}, we have
\begin{equation}\label{eq:base-case-main-3}
	\psi_{G\wr S_2}(\mathcal{G}_2,\NYGD_{i,j})\geq\min\left(\psi_{G\wr S_2}(\mathcal{G}_2,\NYGD_{1,i}),\psi_{G\wr S_2}(\mathcal{G}_2,\NYGD_{1,j})\right),\text{ for }1<i<j\leq t.
\end{equation}
Also, from \eqref{eq:base-case-main-1}, \eqref{eq:base-case1.4}, and \eqref{eq:base-case1.5-1}, we have
\begin{equation}\label{eq:base-case-main-4}
	\psi_{G\wr S_2}(\mathcal{G}_2,\NYGD_{k})\geq \psi_{G\wr S_2}(\mathcal{G}_2,\NYGD_{1,k}),\text{ for }1<k\leq t.
\end{equation}

Now, using \eqref{eq:sp-gap-rep1} and \eqref{eq:base-case0}, we have
\begin{align}
	\psi_{G\wr S_2}(\mathcal{G}_2):=\underset{\mu\in\y_2(\irrG)}{\min}\big\{\psi_{G\wr S_2}(\mathcal{G}_2,\mu)\big\}=&\;\underset{\substack{1\leq i<j\leq t\\1\leq k\leq t}}{\min}\big\{\psi_{G\wr S_2}(\mathcal{G}_2,\NYGD_{i,j}),\psi_{G\wr S_2}(\mathcal{G}_2,\NYGD_k),\psi_{G\wr S_2}(\mathcal{G}_2,\NYGD_k')\big\}\nonumber\\
	=&\;\underset{1< k\leq t}{\min}\big\{\psi_{G\wr S_2}(\mathcal{G}_2,\NYGD_{1,k}),\psi_{G\wr S_2}(\mathcal{G}_2,\NYGD_1),\psi_{G\wr S_2}(\mathcal{G}_2,\NYGD_1')\big\}\label{eq:base-case1.6-1}\\
	=&\;\psi_{G\wr S_2}(\mathcal{G}_2,\mathcal{R}_2).\label{eq:base-case1.6-2}
\end{align}
The equality in \eqref{eq:base-case1.6-1} follows from \eqref{eq:base-case-main-2},\eqref{eq:base-case-main-3}, and \eqref{eq:base-case-main-4}. Also, the equality in \eqref{eq:base-case1.6-2} follows from Theorem \ref{thm:grp_action-decomposition}. The proof is completed here.\qedhere
\end{proof}

Before proceeding further, let us discuss one useful result for any finite group. Given a finite group $\group$, let us recall $\mathbb{C}[\group]^{(\s)}\subset \mathbb{C}[G\wr S_n]$ from \eqref{eq:Re-Sym_Gp-Alg0}, $\Delta_{\group}(\mathcal{W},\rho)$ from \eqref{eq:delta-op}, and $\mathbb{R}_+[\group]^{(\s)}$ from \eqref{eq:Re-Sym_Gp-Alg},
and define
\[\Gamma\left(\group\right):=\bigg\{\mathcal{W}\in\mathbb{C}[\group]^{(\s)}:\Delta_{\group}(\mathcal{W},\rho)\text{ is positive semidefinite for all representation }\rho\bigg\}.\]
\begin{lem}[{\cite[Proposition 4.1]{Cesi}}]\label{lem:Cesi's_useful_result}
	Let $\group$ be a finite group. Then, for $\Gamma\left(\group\right)$, we have
	\begin{enumerate}
		\item If $\mathcal{W},\mathcal{Z}\in\Gamma\left(\group\right)$, then $\alpha\cdot\mathcal{W}+\beta\cdot\mathcal{Z}\in\Gamma\left(\group\right)$ for any two positive real numbers $\alpha,\beta>0$.
		\item $\mathbb{R}_+[\group]^{(\s)}\subset\Gamma\left(\group\right)$.
		\item If $\mathscr{H}$ is a subgroup of $\group$, then $\Gamma\left(\mathscr{H}\right)\subset\Gamma\left(\group\right)$.
	\end{enumerate}
\end{lem}
\begin{prop}[{\cite[Proposition 3.2]{Cesi-octopus}}]\label{prop:Cesi's_useful_resul}
	Let $\group$ be a finite group and $\mathscr{H}$ be a subgroup of $\group$. If $\mathcal{W}\in\mathbb{R}_+[\group]^{(\s)}$ and $\mathcal{Z}\in\mathbb{R}_+[\mathscr{H}]^{(\s)}$ are such that $\mathcal{W}-\mathcal{Z}\in\Gamma\left(\group\right)$, then 
	\begin{align*}
		&\psi_{\group}(\W)\\
		\geq&\;\min\Big\{\psi_{\mathscr{H}}(\mathcal{Z}),\min\big\{\psi_{\group}(\W,\rho):\rho\in\widehat{\group},\;\rho\downarrow^{\group}_{\mathscr{H}}\text{ contains the trivial representation of }\mathscr{H}\big\}\Big\}
	\end{align*}
\end{prop}
Given a matrix $M$ of size $m\times m$, let $\Sp(M)$ denotes the spectrum of $M$. Also recall that we can list down its eigenvalues as follows:
\[\lambda_1(M)\leq\lambda_2(M)\leq\cdots\leq\lambda_m(M).\]
i.e., $\Sp(M)=\{\lambda_i(M):1\leq i\leq m\}$. Let us first prove a useful lemma as follows:
\begin{lem}\label{lem:ev_conn_w_tensor_prod}
	Let $X$ and $Y$ be two self-adjoint matrices of the same order. Also, let 
	\[\diag(a_1,\dots,a_k):=\begin{pmatrix}
		a_1&&\\
		&\ddots&\\
		&&a_k
	\end{pmatrix},\]
be the $k\times k$ diagonal matrix with real entries, and $I_k$ be the identity matrix of order $k$. Then,
\begin{equation}\label{eq:ev_conn_w_tensor_prod}
	\Sp(X\otimes I_k+Y\otimes\diag(a_1,\dots,a_k))=\underset{1\leq i\leq k}{\cup}\Sp(X+a_i\cdot Y).
\end{equation}
Moreover, if $Y$ is positive semidefinite, then 
\[\lambda_1(X\otimes I_k+Y\otimes\diag(a_1,\dots,a_k))=\lambda_1\left(X+\underset{1\leq i\leq k}{\min}\{a_i\}\cdot Y\right).\]
\end{lem}
\begin{proof}
	We first note that $X\otimes I_k+Y\otimes\diag(a_1,\dots,a_k)$ and $X+a_i\cdot Y,\;1\leq i\leq k$ are self-adjoint matrices, i.e., their eigenvalues are all real. Let $\lambda\in\Sp(X+a_1\cdot Y)\cup\dots\cup\Sp(X+a_k\cdot Y)$. Then, $\lambda\in\Sp(X+a_i\cdot Y)$ for some $1\leq i\leq k$. Suppose $v$ is the eigenvector (column vector) of $X+a_i\cdot Y$ with eigenvalue $\lambda$, i.e., $Xv+a_i\cdot Yv=\lambda \cdot v$. Set $e_i$ to be the elementary column vector of size $k\times 1$ with the only non-zero entry $1$ at position $(i,1)$. Then, we have
	\begin{align*}
		&(X\otimes I_k+Y\otimes\diag(a_1,\dots,a_k))\cdot (v\otimes e_i)\\
		=&Xv\otimes e_i+Yv\otimes a_i\cdot e_i=(Xv+a_i\cdot Yv)\otimes e_i=\lambda \cdot v\otimes e_i=\lambda\cdot (v\otimes e_i).
	\end{align*}
	Thus, $\lambda$ is an eigenvalue of $X\otimes I_k+Y\otimes\diag(a_1,\dots,a_k)$ with eigenvector $v\otimes e_i$. Therefore,
	\[\underset{1\leq i\leq k}{\cup}\Sp(X+a_i\cdot Y)\subseteq\Sp(X\otimes I_k+Y\otimes\diag(a_1,\dots,a_k));\]
	but the two sets have the same cardinality. Hence, we have \eqref{eq:ev_conn_w_tensor_prod}.
	
	Now assume that $Y$ is positive semidefinite. Then, $\left(a_j-\underset{1\leq i\leq k}{\min}\{a_i\}\right)\cdot Y$ is also positive semidefinite for every $1\leq j\leq k$. Therefore, using \cite[Corollary 4.3.3]{HJ-mat.-an.}, we can conclude that
	\[\lambda_1(X+a_j\cdot Y)=\lambda_1\left(X+\underset{1\leq i\leq k}{\min}\{a_i\}\cdot Y+\left(a_j-\underset{1\leq i\leq k}{\min}\{a_i\}\right)\cdot Y\right)\geq \lambda_1\left(X+\underset{1\leq i\leq k}{\min}\{a_i\}\cdot Y\right).\]
	Thus, the proof follows from \eqref{eq:ev_conn_w_tensor_prod}.
\end{proof}

Let us now focus on our case and recall $\mathcal{G}_n\in\mathbb{C}[G\wr S_n]$ from \eqref{eq:group_alg_sett}. Without loss of generality, from now on, we assume $y_n=\min\{y_i:1\leq i\leq n\}$. This assumption is permissible because $\psi_{G\wr S_n}(\transp\mathcal{G}_n\transp)=\psi_{G\wr S_n}(\mathcal{G}_n)$ and $\psi_{G\wr S_n}(\transp\mathcal{G}_n\transp,\mathcal{R}_n)=\psi_{G\wr S_n}(\mathcal{G}_n,\mathcal{R}_n)$ for any transposition $\transp$. We now define a projection $\vartheta\mathcal{G}_n\in\mathbb{C}[G\wr S_{n-1}]$ of $\mathcal{G}_n\in\mathbb{C}[G\wr S_n]$ as follows:
\begin{equation}\label{eq:projection-defn}
	\vartheta\mathcal{G}_n:=\sum_{1\leq u<v< n}\left(x_{u,v}+\frac{x_{u,n}x_{v,n}}{x_{1,n}+\cdots+x_{n-1,n}}\right)(u,v)_G+\sum_{w=1}^{n-1}y_w\sum_{g\in G}\alpha_g\left(g\right)^{(w)}.
\end{equation}
\begin{thm}\label{thm:inductive-step}
	For the representation $\mathcal{R}_n$ of $G\wr S_n$ defined in \eqref{eq:grp_action-repn}, $\mathcal{G}_n\in\mathbb{C}[G\wr S_n]$ defined in \eqref{eq:group_alg_sett}, and the projection $\vartheta\mathcal{G}_n$ defined in \eqref{eq:projection-defn}, we have
	\[\psi_{G\wr S_n}(\mathcal{G}_n,\mathcal{R}_n)\leq \psi_{G\wr S_{n-1}}(\vartheta\mathcal{G}_n,\mathcal{R}_{n-1}).\]
\end{thm}
\begin{proof}
	Let us recall that $t:=|\irrG|$ and $\irrG:=\{\sigma_1,\dots,\sigma_t\}$, where $\sigma_1=\trivial$, the trivial representation of $G$. We write $\mu\in\y_n(\irrG)$ as the tuple $(\mu^{(1)},\dots,\mu^{(t)})$, where $\mu^{(i)}:=\mu(\sigma_i)$ for each $1\leq i\leq t$. $W^{\sigma_i}$ is the irreducible $G$-module corresponding to $\sigma_i$ and $d_i=\dim(W^{\sigma_i})$ for each $1\leq i\leq t$. Then, from Theorem \ref{thm:grp_action-decomposition}, the representation $\mathcal{R}_n$ decomposes into irreducible $G\wr S_n$-modules as follows:
	\begin{equation}\label{eq:IS-0.1}
		\mathbb{C}[G\times [n]]\cong \;V^{\MYGD_0}\oplus d_1\;V^{\MYGD_1}\oplus d_2\;V^{\MYGD_2}\oplus\cdots\oplus d_t\;V^{\MYGD_t},
	\end{equation}
	 where $\MYGD_0$ and $\MYGD_1$ are defined in \eqref{eq:def-action-not}, and $\MYGD_j=\MYGD_{1,j}$ (see \eqref{eq:notations-YGD}) for $2\leq j\leq t$. Theorem \ref{thm:grp_action-decomposition} also provides the decomposition of the representation $\mathcal{R}_{n-1}$ (replace $n$ by $n-1$) into irreducible $G\wr S_{n-1}$-modules as follows:
	 \begin{equation}\label{eq:IS-0.2}
	 	\mathbb{C}[G\times [n-1]]\cong \;V^{\MYGD_0'}\oplus d_1\;V^{\MYGD_1'}\oplus d_2\;V^{\MYGD_2'}\oplus\cdots\oplus d_t\;V^{\MYGD_t'},\text{ where }
	 \end{equation}
	 \begin{equation}\label{eq:IS-0.3}
	 	\begin{cases}
	 		\begin{split}
	 		&\hspace{0.25cm}(n-1)\text{ boxes}\\[-1ex]
	 		\MYGD_{0}':=&\tiny{\left(\begin{array}{c}\overbrace{\young(\;\;{\;$\cdots$\;\;}\;\;)}\end{array},\emptyset,\dots,\emptyset\right)} \in\y_{n-1}(\irrG),\\[0.5ex]
	 		&\hspace{0.25cm}(n-2)\text{ boxes}\\[-1ex]
	 		\MYGD_{1}':=&\tiny{\left(\begin{array}{c}\overbrace{\young(\;\;{\;$\cdots$\;\;}\;\;,\;)}\end{array},\emptyset,\dots,\emptyset\right)} \in\y_{n-1}(\irrG),\\[0.5ex]
	 		&\hspace{0.35cm}(n-2)\text{ boxes}\\[-1ex]
	 		\MYGD_{j}':=&\tiny{\left(\begin{array}{c}\overbrace{\young(\;\;{\;$\cdots$\;\;}\;\;)}\end{array},\emptyset,\dots,\emptyset,\underset{\uparrow}{\begin{array}{c}\yng(1)\end{array}},\emptyset,\dots,\emptyset\right)} \in\y_{n-1}(\irrG),\;\;2\leq j\leq t.\\[-1ex]
	 		&\hspace{3.75cm}j\text{th position.}
	 	\end{split}
	 	\end{cases}
	 \end{equation}
	 The decompositions \eqref{eq:IS-0.1} and \eqref{eq:IS-0.2} together imply the theorem if the following inequalities hold:
	 \begin{align}
	 	\psi_{G\wr S_n}(\mathcal{G}_n,\MYGD_{0}\oplus\MYGD_{1})&\leq \psi_{G\wr S_{n-1}}(\vartheta\mathcal{G}_n,\MYGD_{0}'\oplus\MYGD_{1}'),\;\;\text{ and }	\label{eq:IS-1.1}\\
	 	\psi_{G\wr S_n}(\mathcal{G}_n,\MYGD_j)&\leq \psi_{G\wr S_{n-1}}(\vartheta\mathcal{G}_n,\MYGD_j')\;\;\text{ for }2\leq j\leq t.\label{eq:IS-1.2}
	 \end{align}
	 Let us choose an arbitrary integer $j$ satisfying $2\leq j\leq t$. Throughout this theorem, we will use the ordered basis $\B_j$ of $\MYGD_{j}$ given in Proposition \ref{prop:permutation-matrix}. Now, using Theorem \ref{thm:GT-decomposition} and notation \eqref{eq:diagonal_matrix}, we have
	 \begin{equation}\label{eq:IS-2.0}
	 \Delta_{G\wr S_n}\left(y_w\displaystyle\sum_{g\in G}\alpha_g\left(g\right)^{(w)},\MYGD_{j}\right)=\underset{1\leq i\leq n}{\oplus}\delta_{i,w}y_i\D(\sigma_j),
	 \end{equation}
	 for every $w\in [n]$. Thus,
	 \begin{equation}\label{eq:IS-2.1}
	 	 \Delta_{G\wr S_n}\left(\sum_{w=1}^{n}y_w\displaystyle\sum_{g\in G}\alpha_g\left(g\right)^{(w)},\MYGD_{j}\right)
	 =\begin{pmatrix}
	 	y_1& & & &\\
	 	&y_2& & & \\
	 	& &\ddots& &\\
	 	& & & & y_n
	 \end{pmatrix}\otimes
	 	\begin{pmatrix}
	 		D^{\sigma_j}_1& & & &\\
	 		&D^{\sigma_j}_2& & & \\
	 		& &\ddots& &\\
	 		& & & & D^{\sigma_j}_{d_j} 
	 	\end{pmatrix}.
	 \end{equation}
	 Now, using Proposition \ref{prop:permutation-matrix}, we also have
	 \begin{equation}\label{eq:IS-2.2}
	 	 \Delta_{G\wr S_n}\left(\sum_{1\leq u<v\leq n}x_{u,v}(u,v)_G,\MYGD_{j}\right)=\sum_{1\leq u<v\leq n}x_{u,v}\left(I_n-\big[(u,v)\big]_n\right)\otimes I_{d_j}=M_n\otimes I_{d_j},
	 \end{equation}
 where we set $M_n$ to be the $n\times n$ matrix $\displaystyle\sum_{1\leq u<v\leq n}x_{u,v}\left(I_n-\big[(u,v)\big]_n\right)$. Therefore, combining \eqref{eq:IS-2.1} and \eqref{eq:IS-2.2}, we have
 \begin{equation}\label{eq:IS-2.3}
 	\Delta_{G\wr S_n}\left(\mathcal{G}_n,\MYGD_{j}\right)=M_n\otimes I_{d_j}+\begin{pmatrix}
 		y_1& & & &\\
 		&y_2& & & \\
 		& &\ddots& &\\
 		& & & & y_n
 	\end{pmatrix}\otimes
 	\begin{pmatrix}
 	D^{\sigma_j}_1& & & &\\
 	&D^{\sigma_j}_2& & & \\
 	& &\ddots& &\\
 	& & & & D^{\sigma_j}_{d_j} 
	\end{pmatrix}.
 \end{equation}
 Again, Proposition \ref{prop:def-action} implies that
 \begin{equation}\label{eq:IS-2.4}
 	\Delta_{G\wr S_n}\left(\mathcal{G}_n,\MYGD_0\oplus\MYGD_1\right) =\sum_{1\leq u<v\leq n}x_{u,v}\left(I_n-\big[(u,v)\big]_n\right)=M_n.
 \end{equation}
We note that Proposition \ref{prop:permutation-matrix} and Proposition \ref{prop:def-action} hold for the case of $G\wr S_{n-1}$ (replace $n$ by $n-1$, and $\MYGD_i$ by $\MYGD_i'\;0\leq i\leq t$). Therefore, by a similar argument given for \eqref{eq:IS-2.3} and \eqref{eq:IS-2.4}, we have
\begin{align}
	\Delta_{G\wr S_{n-1}}\left(\vartheta\mathcal{G}_n,\MYGD_0'\oplus\MYGD_1'\right) &=\sum_{1\leq u<v<n}\left(x_{u,v}+\frac{x_{u,n}x_{v,n}}{x_{1,n}+\cdots+x_{n-1,n}}\right)\left(I_{n-1}-\big[(u,v)\big]_{n-1}\right),\nonumber\\
	&=:M^{\vartheta}_{n-1}\label{eq:IS-2.5}\\
	\Delta_{G\wr S_{n-1}}\left(\vartheta\mathcal{G}_{n},\MYGD_{j}'\right)&=M^{\vartheta}_{n-1}\otimes I_{d_j}+
	\begin{pmatrix}
		y_1& & & &\\
		&y_2& & & \\
		& &\ddots& &\\
		& & & & y_{n-1}
	\end{pmatrix}\otimes\D(\sigma_j).\label{eq:IS-2.6}
\end{align}
Given a matrix $M$ of size $m\times m$, let us recall that we can list down its eigenvalues as follows:
\[\lambda_1(M)\leq\lambda_2(M)\leq\cdots\leq\lambda_m(M).\]
We now set $D_j=\min\{D_i^{\sigma_j}:1\leq i\leq d_j\}$
\begin{align*}
	F_n=\begin{pmatrix}
		y_1& & & &\\
		&y_2& & & \\
		& &\ddots& &\\
		& & & y_{n-1}& \\
		& & & & y_n
	\end{pmatrix}\text
	{ and }
	F^{\vartheta}_{n-1}&=\begin{pmatrix}
		y_1& & & &\\
		&y_2& & & \\
		& &\ddots& &\\
		& & & y_{n-1} & \\
	\end{pmatrix}.
\end{align*}
In view of \eqref{eq:IS-2.3}, \eqref{eq:IS-2.6}, and Lemma \ref{lem:ev_conn_w_tensor_prod}, we can conclude that
\begin{align}
	\lambda_1\left(\Delta_{G\wr S_n}\left(\mathcal{G}_n,\MYGD_{j}\right)\right)&=\lambda_1\left(M_n\otimes I_{d_j}+F_n\otimes\D(\sigma_j)\right)=\lambda_1\left(M_n+D_j\cdot F_n\right)\label{eq:IS-2.7}\\
	\lambda_1\left(\Delta_{G\wr S_{n-1}}\left(\vartheta\mathcal{G}_n,\MYGD_{j}'\right)\right)&=\lambda_1\left(M_{n-1}^{\vartheta}\otimes I_{d_j}+F^{\vartheta}_{n-1}\otimes\D(\sigma_j)\right)\label{eq:IS-2.8}\\
		&=\lambda_1\left(M_{n-1}^{\vartheta}+D_j\cdot F_{n-1}^{\vartheta}\right)\nonumber
\end{align}
We note that the matrix $M_{n-1}^{\vartheta}\oplus[0]_{1\times 1}$ has size $n\times n$, and
\begin{align}
	L_n:=&M_n-M_{n-1}^{\vartheta}\oplus[0]_{1\times 1}\nonumber\\
	=&\sum_{u=1}^{n-1}x_{u,n}\left(I_n-\big[(u,n)\big]_n\right)-\sum_{1\leq u<v<n}\frac{x_{u,n}x_{v,n}}{x_{1,n}+\cdots+x_{n-1,n}}\left(I_{n}-\big[(u,v)\big]_{n-1}\oplus[1]_{1\times 1}\right).\label{eq:IS-3.1}
\end{align}
The $(i,j)$th entry of  $L_n$ is given as follows:
\begin{equation}\label{eq:IS-rank-one}
	\left(L_n\right)_{ij}=\begin{cases}
		\frac{x_{i,n}x_{j,n}}{x_{1,n}+\cdots+x_{n-1,n}}&\text{ if }i\neq j<n\\
		\frac{x_{i,n}^2}{x_{1,n}+\cdots+x_{n-1,n}}&\text{ if }i = j<n\\
		-x_{i,n}&\text{ if }i < j = n\\
		-x_{j,n}&\text{ if }j <i = n\\
		\displaystyle\sum_{u=1}^{n-1}x_{u,n}&\text{ if }j =i = n\\
	\end{cases}
\end{equation}
Let us set $A=x_{1,n}+\cdots+x_{n-1,n}$ and $A_i=\frac{x_{i,n}}{\sqrt{A}}$ for $1\leq i\leq n-1$. Then $L_n$ can be written as
\[L_n=\begin{pmatrix}
	A_1^2&A_1A_2&\cdots&\cdots&\cdots&A_1A_{n-1}&-A_1\sqrt{A}\\
	A_2A_1&A_2^2&\cdots&\cdots&\cdots&A_2A_{n-1}&-A_2\sqrt{A}\\
	\vdots&\vdots&\vdots&\vdots&\vdots&\vdots&\vdots\\
	A_iA_{1}&A_iA_2&\cdots&A_iA_j&\cdots&A_iA_{n-1}&-A_i\sqrt{A}\\
	\vdots&\vdots&\vdots&\vdots&\vdots&\vdots&\vdots\\
	A_{n-1}A_{1}&A_{n-1}A_2&\cdots&A_{n-1}A_j&\cdots&A_{n-1}A_{n-1}&-A_{n-1}\sqrt{A}\\
	-A_1\sqrt{A}&-A_2\sqrt{A}&\cdots&-A_j\sqrt{A}&\cdots&-A_{n-1}\sqrt{A}&A
\end{pmatrix},\]
which clearly has rank $1$. Therefore, using \eqref{eq:IS-3.1} and \cite[Corollary 4.3.5]{HJ-mat.-an.}, we have
\begin{align}
	\lambda_2(M_n)&=\lambda_2\left(M^{\vartheta}_{n-1}\oplus[0]_{1\times 1}+L_n\right)\leq\lambda_3\left(M^{\vartheta}_{n-1}\oplus[0]_{1\times 1}\right)=\lambda_2\left(M^{\vartheta}_{n-1}\right)\label{eq:IS-4.1}\\
	\lambda_1(M_n+D_j\cdot F_n)&=\lambda_1\left(M^{\vartheta}_{n-1}\oplus[0]_{1\times 1}+L_n+D_j\cdot F_n\right)\nonumber\\
	&\leq \lambda_2\left(M^{\vartheta}_{n-1}\oplus[0]_{1\times 1}+D_j\cdot F_n\right)\label{eq:IS-4.2}
\end{align}
Now, recalling the notation $\Sp(M)$ for the spectrum of a given matrix $M$, we obtain
\begin{align}\label{eq:IS-5.1}
	\Sp\left(M^{\vartheta}_{n-1}\oplus[0]_{1\times 1}+D_j\cdot F_n\right)&=\Sp\left(\left(M^{\vartheta}_{n-1}+D_j\cdot F^{\vartheta}_{n-1}\right)\oplus[D_jy_n]_{1\times 1}\right)\\
	&=\Sp\left(M^{\vartheta}_{n-1}+D_j\cdot F^{\vartheta}_{n-1}\right)\cup\{D_jy_n\}.\nonumber
\end{align}
The matrix $M^{\vartheta}_{n-1}\oplus[0]_{1\times 1}$ is positive semidefinite using the second part of Lemma \ref{lem:Cesi's_useful_result},
\[\sum_{1\leq u<v<n}\left(x_{u,v}+\frac{x_{u,n}x_{v,n}}{x_{1,n}+\cdots+x_{n-1,n}}\right)(u,v)_G\in\mathbb{R}_+[G\wr S_n]^{(\s)},\]
and \eqref{eq:IS-2.5}. Therefore \cite[Corollary 4.3.3]{HJ-mat.-an.} implies
\begin{equation}\label{eq:IS-5.2}
	\lambda_k\left(M^{\vartheta}_{n-1}\oplus[0]_{1\times 1}+D_j\cdot F_n\right)\geq \lambda_k\left(D_j\cdot F_n\right)\text{ for all }k=1,\dots,n,
\end{equation}
i.e., in particular, using $y_n=\min\{y_i:1\leq i\leq n\}$, we have
\[\lambda_1\left(M^{\vartheta}_{n-1}\oplus[0]_{1\times 1}+D_j\cdot F_n\right)\geq \lambda_1\left(D_j\cdot F_n\right)=D_jy_n.\]
But \eqref{eq:IS-5.1} says that $D_jy_n$ is an eigenvalue of $M^{\vartheta}_{n-1}\oplus[0]_{1\times 1}+D_j\cdot F_n$, so it must be the lowest one. Again, using \eqref{eq:IS-5.1}, we conclude that
\begin{equation}\label{eq:IS-6}
	\lambda_2\left(M^{\vartheta}_{n-1}\oplus[0]_{1\times 1}+D_j\cdot F_n\right)=\lambda_1\left(M^{\vartheta}_{n-1}+D_j\cdot F^{\vartheta}_{n-1}\right)
\end{equation}

Now, we are ready to complete the proof. We first prove \eqref{eq:IS-1.1}. The fact $\MYGD_{0}\oplus\MYGD_{1}$ (respectively $\MYGD_{0}'\oplus\MYGD_{1}'$) contains the trivial representation of $G\wr S_n$ (respectively $G\wr S_{n-1}$) implies
\begin{align*}
	0&=\lambda_1\left(\Delta_{G\wr S_n}(\mathcal{G}_n,\MYGD_{0}\oplus\MYGD_{1})\right)<\lambda_2\left(\Delta_{G\wr S_n}(\mathcal{G}_n,\MYGD_{0}\oplus\MYGD_{1})\right),\text{ uses the irreducibility},\\
	0&=\lambda_1\left(\Delta_{G\wr S_{n-1}}(\vartheta\mathcal{G}_n,\MYGD_{0}'\oplus\MYGD_{1}')\right)\leq\lambda_2\left(\Delta_{G\wr S_{n-1}}(\vartheta\mathcal{G}_n,\MYGD_{0}'\oplus\MYGD_{1}')\right)
\end{align*}
Therefore,
\begin{align}\label{eq:IS-6.1}
	\psi_{G\wr S_n}(\mathcal{G}_n,\MYGD_{0}\oplus\MYGD_{1})&=\lambda_2\left(\Delta_{G\wr S_n}(\mathcal{G}_n,\MYGD_{0}\oplus\MYGD_{1})\right)\nonumber\\
	&=\lambda_2(M_n),\text{ using \eqref{eq:IS-2.4}}\nonumber\\
	&\leq \lambda_2\left(M^{\vartheta}_{n-1}\right),\text{ by \eqref{eq:IS-4.1}}\nonumber\\
	&=\lambda_2\left(\Delta_{G\wr S_{n-1}}\left(\vartheta\mathcal{G}_n,\MYGD_0'\oplus\MYGD_1'\right)\right),\text{ by \eqref{eq:IS-2.5}}\nonumber\\
	&\leq \psi_{G\wr S_{n-1}}\left(\vartheta\mathcal{G}_n,\MYGD_0'\oplus\MYGD_1'\right)
\end{align}
The inequality in \eqref{eq:IS-6.1} follows from the fact $\lambda_1\left(\Delta_{G\wr S_{n-1}}(\vartheta\mathcal{G}_n,\MYGD_{0}'\oplus\MYGD_{1}')\right)=0$, and hence \eqref{eq:IS-1.1} is proved. We now focus on proving \eqref{eq:IS-1.2}. For $2\leq j\leq t$, we have,
\begin{align}\label{eq:IS-6.2}
		\psi_{G\wr S_n}(\mathcal{G}_n,\MYGD_j)&=\lambda_1\left(\Delta_{G\wr S_n}(\mathcal{G}_n,\MYGD_j)\right),\text{ by irreducibility}\nonumber\\
		&=\lambda_1\left(M_n+D_j\cdot F_n\right),\text{ using \eqref{eq:IS-2.7}}\nonumber\\
		&\leq \lambda_2\left(M^{\vartheta}_{n-1}\oplus[0]_{1\times 1}+D_j\cdot F_n\right),\text{ by \eqref{eq:IS-4.2}}\nonumber\\
		&=\lambda_1\left(M^{\vartheta}_{n-1}+D_j\cdot F^{\vartheta}_{n-1}\right),\text{ by \eqref{eq:IS-6}}\nonumber\\
		&=\lambda_1\left(\Delta_{G\wr S_{n-1}}\left(\vartheta\mathcal{G}_{n},\MYGD_{j}'\right)\right),\text{ using \eqref{eq:IS-2.8}}\nonumber\\
		&\leq\psi_{G\wr S_{n-1}}\left(\vartheta\mathcal{G}_{n},\MYGD_{j}'\right).
\end{align}
This completes the proof of \eqref{eq:IS-1.2}. Therefore, as we mentioned before, the theorem follows from the decompositions \eqref{eq:IS-0.1} and \eqref{eq:IS-0.2} by using \eqref{eq:IS-6.1} and \eqref{eq:IS-6.2}.
 \end{proof}
\begin{lem}[{The octopus inequality}]\label{lem:Octopus}
	For $\mathcal{G}_n\in\mathbb{C}[G\wr S_n]$ defined in \eqref{eq:group_alg_sett}, and the projection $\vartheta\mathcal{G}_n$ defined in \eqref{eq:projection-defn}, we have
	\[\mathcal{G}_n-\vartheta\mathcal{G}_n\in\Gamma(G\wr S_n).\]
\end{lem}
\begin{proof}
	We use \cite[Theorem 2.3]{CLR}, the``octopus inequality", which implies that
	\begin{align*}
		\sum_{u=1}^{n-1}x_{u,n}(u,v)-\sum_{1\leq u<v<n}\frac{x_{u,n}x_{v,n}}{x_{1,n}+\cdots+x_{n-1,n}}\;(u,v)\in\Gamma(S_n),
	\end{align*}
	an alternative algebraic proof is also available at \cite[Section 4]{Cesi-octopus}. Therefore, using the embedding $S_n:\pi\mapsto \pi_G\in G\wr S_n$, and the third part of Lemma \ref{lem:Cesi's_useful_result}, we have
	\begin{align*}
		\sum_{u=1}^{n-1}x_{u,n}(u,v)_G-\sum_{1\leq u<v<n}\frac{x_{u,n}x_{v,n}}{x_{1,n}+\cdots+x_{n-1,n}}\;(u,v)_G\in\Gamma(G\wr S_n).
	\end{align*}
	Again, the first part of Lemma \ref{lem:Cesi's_useful_result} implies that $\displaystyle\sum_{w=1}^{n-1}y_w\displaystyle\sum_{g\in G}\alpha_g\left(g\right)^{(w)}\in\Gamma(G\wr S_n)$; and hence by the same argument, we have
	\[\sum_{w=1}^{n-1}y_w\sum_{g\in G}\alpha_g\left(g\right)^{(w)}+\sum_{u=1}^{n-1}x_{u,n}(u,v)_G-\sum_{1\leq u<v<n}\frac{x_{u,n}x_{v,n}}{x_{1,n}+\cdots+x_{n-1,n}}\;(u,v)_G\in\Gamma(G\wr S_n)\]
	Thus, the lemma follows from \eqref{eq:group_alg_sett} and \eqref{eq:projection-defn}.
\end{proof}
\begin{proof}[\textbf{Proof of Theorem \ref{thm:main_thm}}]
	We use the first principle of mathematical induction for the proof. The base case of $n=2$ follows from Theorem \ref{thm:base-case}. We assume Theorem \ref{thm:main_thm} holds for a positive integer $n-1$. 
	
	Recall that $\G=(V,E)$ is the underlying graph of $\mathcal{G}_n$, i.e., we perform the interchange process on $\G$. Also, recall $\vartheta\mathcal{G}_n$ from \eqref{eq:projection-defn} and denote its underlying graph by $\vartheta\G$. We note that $\Supp\left(\vartheta\mathcal{G}_n\right)$ generates $G\wr S_{n-1}$ if and only if $\vartheta\G$ is connected. Let $u,v$ be any two vertices of $\vartheta\G$. Then the connectedness of $\G$ implies that there is a path
	\[\gamma:u=u_0\rightarrow u_1\rightarrow\cdots\rightarrow u_{\ell}=v\]
	from $u$ to $v$ in $\G$. If $n\notin \gamma$, then $\gamma$ lies in $\vartheta\G$. If $n\in \gamma$, then $u_r=n$ for some positive integer $r$, and we choose the path
	\[\gamma':u=u_0\rightarrow u_1\rightarrow u_{r-1}\overset{\cancel{n}}{\rightarrow}u_{r+1}\rightarrow\cdots\rightarrow u_{\ell}=v\]
	from $u$ to $v$ that lies in $\vartheta\G$. The existence of $\gamma'$ is guaranteed by the definition \eqref{eq:projection-defn} of $\vartheta\mathcal{G}_n$. Hence, $\vartheta\G$ is connected, i.e., $\Supp(\vartheta\mathcal{G}_n)$ generates $G\wr S_{n-1}$. Therefore, using the induction hypothesis, we have the following:
	\begin{equation}\label{eq:main_thm-ind-st}
		\psi_{G\wr S_{n-1}}(\vartheta\mathcal{G}_n)= \psi_{G\wr S_{n-1}}(\vartheta\mathcal{G}_n,\mathcal{R}_{n-1}).
	\end{equation}
	Here, the notations have the same meaning as those in \eqref{eq:sp-gap-rep1} and \eqref{eq:grp_action-repn}.

	Again, recalling $\MYGD_{0}$ and $\MYGD_{1}$ from \eqref{eq:def-action-not}, $\MYGD_{j}:=\MYGD_{1,j}\;(2\leq j\leq t)$ from \eqref{eq:notations-YGD}, and the branching rule Theorem \ref{thm:Branching-rule}, we can conclude that 
	\[\big\{\rho\in\widehat{G\wr S_n}:\rho\downarrow_{G\wr S_{n-1}}^{G\wr S_n}\text{ contains the trivial representation of }G\wr S_{n-1}\big\}=\big\{\MYGD_j:0\leq j\leq t\big\}.\]
	Therefore, using Lemma \ref{lem:Octopus}, i.e., $\mathcal{G}_n-\vartheta\mathcal{G}_n\in\Gamma(G\wr S_n)$, and Proposition \ref{prop:Cesi's_useful_resul}, we have
	\begin{align}
		\psi_{G\wr S_{n}}(\mathcal{G}_n)&\geq\min\{\psi_{G\wr S_{n-1}}\left(\vartheta\mathcal{G}_n\right),\psi_{G\wr S_n}(\mathcal{G}_n,\MYGD_j):0\leq j\leq t\}\nonumber\\
		&=\min\{\psi_{G\wr S_{n-1}}\left(\vartheta\mathcal{G}_n\right),\psi_{G\wr S_n}(\mathcal{G}_n,\mathcal{R}_n)\}\label{eq:main_thm1}\\
		&=\min\{\psi_{G\wr S_{n-1}}(\vartheta\mathcal{G}_n,\mathcal{R}_{n-1}),\psi_{G\wr S_n}(\mathcal{G}_n,\mathcal{R}_n)\}\label{eq:main_thm2}\\
		&=\psi_{G\wr S_n}(\mathcal{G}_n,\mathcal{R}_n)\label{eq:main_thm3}
	\end{align}
	Here, \eqref{eq:main_thm1} follows from the decomposition Theorem \ref{thm:grp_action-decomposition}, \eqref{eq:main_thm2} follows from the inductive step \eqref{eq:main_thm-ind-st}, and \eqref{eq:main_thm3} follows from Theorem \ref{thm:inductive-step}.
	
	On the other hand, \eqref{eq:sp-gap-rep1} immediately implies the other side of the inequality, i.e.,
	\begin{equation}\label{eq:main_thm4}
		\psi_{G\wr S_{n}}(\mathcal{G}_n)\leq \psi_{G\wr S_n}(\mathcal{G}_n,\mathcal{R}_n).
	\end{equation}
	The proof is completed here.
\end{proof}
\section{Example and open question}\label{sec:examples}
In this section, we give an example that shows that Theorem \ref{thm:main_thm} fails when we include a reflection of type \eqref{eq:gen_refl} in the group algebra element \eqref{eq:group_alg_sett}. We first consider the element
\begin{equation}\label{eq:group_alg_sett-ex}
	\mathcal{G}'_n=y\sum_{g\in G}(\Gid,\dots,\Gid,\hspace*{-0.65cm}\underset{\substack{\uparrow\\n\text{th position}}}{g}\hspace*{-0.65cm};\1)+x\sum_{u=1}^{n-1}\sum_{g\in G}(\Gid,\dots,\Gid,\hspace*{-0.65cm}\underset{\substack{\uparrow\\u\text{th position}}}{g}\hspace*{-0.65cm},\Gid,\dots,\Gid,\hspace*{-0.45cm}\underset{\substack{\uparrow\\n\text{th position}}}{g^{-1}}\hspace*{-0.5cm};(u,n)),\;x,y>0.
\end{equation}
We obtain all the eigenvalues of the (left) multiplication action of 
\[\Delta(\mathcal{G}'_n):=y\sum_{g\in G}\left(\Gnid-(\Gid,\dots,\Gid,\hspace*{-0.65cm}\underset{\substack{\uparrow\\n\text{th position}}}{g}\hspace*{-0.65cm};\1)\right)+x\sum_{u=1}^{n-1}\sum_{g\in G}\left(\Gnid-(\Gid,\dots,\Gid,\hspace*{-0.65cm}\underset{\substack{\uparrow\\u\text{th position}}}{g}\hspace*{-0.65cm},\Gid,\dots,\Gid,\hspace*{-0.45cm}\underset{\substack{\uparrow\\n\text{th position}}}{g^{-1}}\hspace*{-0.5cm};(u,n))\right)\]
on $\mathbb{C}[G\wr S_n]$. We now appeal to the \emph{Gelfand-Tsetlin (sub)algebra}, a maximal commuting subalgebra of $\mathbb{C}[G\wr S_n]$ that acts like scalars on all the Gelfand-Tsetlin subspaces of every irreducible $G\wr S_n$-modules. The \emph{Young-Jucys-Murphy} elements generate the Gelfand-Tsetlin (sub)algebra of $\mathbb{C}[G\wr S_n]$; we refer the readers to \cite[Section 3]{MS} for more details. One specific Young-Jucys-Murphy element plays an important role in our context. The Young-Jucys-Murphy elements $X_1(G),\dots,X_n(G)$ of $\mathbb{C}[G\wr S_n]$ are given by
\begin{equation}\label{eq:YJM_els}
	X_1(G)= 0,\;X_i(G)=\displaystyle\sum_{k=1}^{i-1}\sum\limits_{g\in G}(\Gid,\dots,\Gid,\hspace*{-0.65cm}\underset{\substack{\uparrow\\k\text{th position}}}{g}\hspace*{-0.65cm},\Gid,\dots,\Gid,\hspace*{-0.4cm}\underset{\substack{\uparrow\\i\text{th position}}}{g^{-1}}\hspace*{-0.55cm},\Gid,\dots,\Gid;(k,i))\text{ for all }2\leq i\leq n.
\end{equation}
Given $\mu\in\yn(\irrG)$, recall the Gelfand-Tsetlin decomposition
\[V^{\mu}=\underset{T\in\tab_{G}(n,\mu)}{\oplus}V_T\]
from Theorem \ref{thm:GT-decomposition}. Here, $V_T$ is the Gelfand-Tsetlin subspace of $V^{\mu}$, the irreducible $G\wr S_n$-module indexed by $\mu$. Then, the matrices of the action of the Young-Jucys-Murphy elements $X_1(G),\dots,X_n(G)$ on $V_T$ are given as follows (\cite[Theorem 6.5]{MS}):
\begin{equation}\label{eq:YJM_action_on_GT-subspace}
	\big[X_i(G)\big]_{V_T}=\frac{|G|}{\dim(W^{r_T(i)})}c(b_T(i))\cdot I_{\dim(V_T)},
\end{equation}
where $W^{r_T(i)}$ is the irreducible $G$-module indexed by $r_T(i)$, and $I_{\dim(V_T)}$ is the identity matrix of order $\dim(V_T)$. Recall $r_T(i)$ and $c(b_T(i))$ from Definition \ref{def:b_(i)+r_T(i)}.

We note that $\sum_{g\in G}g$ commutes with the elements of $\mathbb{C}[G\wr S_n]$. Thus, the matrix of $\sum_{g\in G}g$ on the irreducible $G$-module $W^{\sigma},\sigma\in\irrG$, can be obtained by Schur's lemma, given below.
\begin{equation}\label{eq:all_sum_action}
	\Bigg[\sum_{g\in G}g\Bigg]_{W^{\sigma}}=\frac{|G|}{\dim(W^{\sigma})}\delta_{\trivial,\sigma}\cdot I_{\dim(W^{\sigma})}.
\end{equation}
\begin{thm}\label{thm:group_alg_sett-ex}
	For every $\mu\in\yn(\irrG)$, the eigenvalues of $\Delta(\mathcal{G}'_n)$ are given by
	\[|G|\left(y+x(n-1)-\frac{y\delta_{\trivial,r_T(n)}+xc(b_T(n))}{\dim(W^{r_T(n)})}\right),\text{ with multiplicity }\dim\left(V^{\mu}\right)\dim\left(V_T\right),\]
	 for all $T\in\tab_{G}(n,\mu)$. Here, $r_T(i)$ and $c(b_T(i))$ are recalled from Definition \ref{def:b_(i)+r_T(i)}, and $W^{r_T(i)}$ is the irreducible $G$-module indexed by $r_T(i)$.
\end{thm}
\begin{proof}
	Given $T\in\tab_{G}(n,\mu)$, we first obtain the eigenvalues of the action of $\Delta(\mathcal{G}'_n)$ on the Gelfand-Tsetlin subspace $V_T$ of the irreducible $G\wr S_n$-module $V^{\mu}$. Theorem \ref{thm:GT-decomposition} implies that 
	\[V_T\cong_{G^n}W^{r_T(1)}\otimes W^{r_T(2)}\otimes \cdots\otimes W^{r_T(n)}.\]
	Therefore, using \eqref{eq:all_sum_action}, we can conclude that
	\begin{equation}\label{eq:group_alg_sett-ex1}
		\Bigg[\sum_{g\in G}(\Gid,\dots,\Gid,\hspace*{-0.65cm}\underset{\substack{\uparrow\\n\text{th position}}}{g}\hspace*{-0.65cm};\1)\Bigg]_{V_T}=\frac{|G|}{\dim(W^{r_T(n)})}\delta_{\trivial,r_T(n)}\cdot I_{\dim(V_T)}.
	\end{equation}
	Recalling the $n$th Young-Jucys-Murphy element $X_n(G)$, we have the following:
	\[\Delta(\mathcal{G}'_n)=(y+x(n-1))|G|\Gnid-y\sum_{g\in G}(\Gid,\dots,\Gid,\hspace*{-0.65cm}\underset{\substack{\uparrow\\n\text{th position}}}{g}\hspace*{-0.65cm};\1)-xX_n(G)\]
	Therefore, using \eqref{eq:YJM_action_on_GT-subspace} and \eqref{eq:group_alg_sett-ex1}, the matrix of the action of $\Delta(\mathcal{G}'_n)$ on the Gelfand-Tsetlin subspace $V_T$ is given below
	\begin{align}\label{eq:group_alg_sett-ex2}
		\big[\Delta(\mathcal{G}'_n)\big]_{V_T}&=\left((y+x(n-1))|G|-\frac{y|G|}{\dim(W^{r_T(n)})}\delta_{\trivial,r_T(n)}-\frac{x|G|}{\dim(W^{r_T(n)})}c(b_T(n))\right)\cdot I_{\dim(V_T)}\nonumber\\
		&=|G|\left(y+x(n-1)-\frac{y\delta_{\trivial,r_T(n)}+xc(b_T(n))}{\dim(W^{r_T(n)})}\right)\cdot I_{\dim(V_T)}
	\end{align}
	Therefore, the theorem follows from the decomposition \eqref{eq:Group_alg._decom.} of the (left) regular representation, the Gelfand-Tsetlin decomposition given in Theorem \ref{thm:GT-decomposition}, and \eqref{eq:group_alg_sett-ex2}.
\end{proof}

Now, recalling $\MYGD_{0},\MYGD_{1}$ from \eqref{eq:def-action-not}, and $\MYGD_{j}=\MYGD_{1,j}\;(j\geq 2)$ from \eqref{eq:notations-YGD}, Theorem \ref{thm:group_alg_sett-ex} implies
\begin{align*}
	\psi_{G\wr S_{n}}(\mathcal{G}'_n,\MYGD_{0}\oplus\MYGD_{1})&=x|G|,\text{ obtained from }T=\tiny{\left(\begin{array}{c}\young(134{\;$\cdots$\;\;}{n},2)\end{array},\emptyset,\dots,\emptyset\right)},\\
	\psi_{G\wr S_{n}}(\mathcal{G}'_n,\MYGD_j)&=x|G|,\text{ obtained from }T=\tiny{\left(\begin{array}{c}\young(**{\;$\cdots$\;\;}{n})\end{array},\emptyset,\dots,\begin{array}{c}\young(*)\end{array},\dots,\emptyset\right)}.
\end{align*}
Therefore, using \eqref{eq:grp_action-repn}, we have 
\begin{equation}\label{eq:example1.1}
	\psi_{G\wr S_{n}}(\mathcal{G}'_n,\mathcal{R}_n)=x|G|
\end{equation}
For $j\geq 2$, we also have
\begin{equation}\label{eq:example1.2}
	\psi_{G\wr S_{n}}\left(\mathcal{G}'_n,\tiny{\left(\emptyset,\dots,\emptyset,\underset{\underset{j\text{th position}}{\uparrow}}{\begin{array}{c}\young(\;{\;$\cdots$\;\;}\;)\end{array}},\emptyset,\dots,\emptyset\right)}\right)=|G|\left(y+x(n-1)\left(1-\frac{1}{\dim(W^{\sigma_j})}\right)\right).
\end{equation}

In particular, if $G$ to be a commutative group (of order $t$), then all the irreducible $G$-modules are one-dimensional. Thus, setting $y<x$ (e.g., $y=2023,x=2024$), we have
\[\psi_{G\wr S_{n}}\left(\mathcal{G}'_n\right)\leq\psi_{G\wr S_{n}}\left(\mathcal{G}'_n,\tiny{\left(\emptyset,\dots,\emptyset,\begin{array}{c}\young(\;{\;$\cdots$\;\;}\;)\end{array},\emptyset,\dots,\emptyset\right)}\right)=y|G|<x|G|=\psi_{G\wr S_{n}}(\mathcal{G}'_n,\mathcal{R}_n)\]
from \eqref{eq:example1.1} and \eqref{eq:example1.2}. Thus, the conclusion of Theorem \ref{thm:main_thm} is not true for this example. Therefore, Theorem \ref{thm:main_thm} fails to hold if we consider reflections of the form \eqref{eq:gen_refl} in the group algebra element \eqref{eq:group_alg_sett}. Below, we ask a natural open question to the complete generality.
\begin{opqn}
	Let $G$ be any finite group and $S_n$ be the symmetric group. Consider the group algebra element $\mathcal{G}''\in\mathbb{C}[G\wr S_n]$ as follows
	\[\mathcal{G}''=\sum_{1\leq u<v\leq n}x_{u,v}\sum_{g\in G}\beta_{g}(\dots,\Gid,\hspace*{-0.65cm}\underset{\substack{\uparrow\\u\text{th position}}}{g}\hspace*{-0.65cm},\Gid,\dots,\Gid,\hspace*{-0.45cm}\underset{\substack{\uparrow\\v\text{th position}}}{g^{-1}}\hspace*{-0.6cm},\dots;(u,v))+\sum_{w=1}^{n}y_w\sum_{g\in G}\alpha_{g}(\dots,\Gid,\hspace*{-0.65cm}\underset{\substack{\uparrow\\w\text{th position}}}{g}\hspace*{-0.65cm},\Gid,\dots;\1),\]
	where $x_{u,v},\beta_g,\alpha_g\geq 0$ and $y_w>0$. Assume that the support of $\mathcal{G}''$, i.e.,
	\[\Bigg\{(\dots,\Gid,\hspace*{-0.65cm}\underset{\substack{\uparrow\\u\text{th position}}}{g}\hspace*{-0.65cm},\Gid,\dots,\Gid,\hspace*{-0.45cm}\underset{\substack{\uparrow\\v\text{th position}}}{g^{-1}}\hspace*{-0.6cm},\dots;(u,v)),(g)^{(w)}:x_{u,v},\beta_{g},\alpha_g>0,1\leq u<v\leq n,g\in G,w\in[n]\Bigg\}\]
	generates $G\wr S_n$. Then obtain the irreducible representation $\rho$ of $G\wr S_n$ such that
	\[\psi_{G\wr S_{n}}(\mathcal{G}'')=\psi_{G\wr S_{n}}(\mathcal{G}'',\rho).\]
\end{opqn}
\subsection*{Acknowledgement} I express my sincere gratitude to Gady Kozma for an insightful discussion, for sharing the article \cite{Cesi}, and for the motivation to explore this direction. I thank Arvind Ayyer for his inspiring comments and for suggesting the title. I also thank Gideon Amir for his encouragement and support and for hosting me as a post-doctoral fellow at Bar-Ilan University. The research is supported by the Israel Science Foundation grant \#957/20.

\bibliography{ACWr_ref}{}
\bibliographystyle{plain}
\end{document}